\documentclass[11pt,reqno]{amsart}
\usepackage{amsfonts}
\usepackage{amsthm, amsfonts, amssymb, color}
\usepackage{mathrsfs}
\usepackage{caption}
\usepackage{enumerate}
\usepackage{cases}
\usepackage{amsmath}
\usepackage{mathbbol}
\usepackage{footnote}
\usepackage{amstext, amsxtra}
\usepackage{txfonts}
\usepackage{tikz}

\usepackage[colorlinks, linkcolor=black, citecolor=blue,   hypertexnames=false]{hyperref}
\usepackage{graphicx}
\usepackage{subfigure}
\usepackage[symbol]{footmisc}
\allowdisplaybreaks
\def\R {\mathbb{R}}
\def\N {\mathbb{N}}

\def\x{\boldsymbol{x}}

\def\d{{\rm d}}

\def \and {{\qquad\text{and}\qquad}}

\newtheorem{theorem}{Theorem}[section]
\newtheorem{proposition}[theorem]{Proposition}
\newtheorem{corollary}[theorem]{Corollary}

\newtheorem{definition}[theorem]{Definition}
\newtheorem{example}[theorem]{Example}
\newtheorem{remark}[theorem]{Remark}

\numberwithin{equation}{section}
\theoremstyle{definition}

\title[HUP and wave equations]
{Heisenberg Uniqueness Pairs and the wave equation}

\author{Shanlin Huang \quad Jiaqi Yu }

\address{Shanlin Huang,  School of Mathematics and Statistics, Hubei Key Laboratory of Engineering Modeling and Scientific Computing, Huazhong University of Science and Technology,  Wuhan,  430074,  P.R. China}
\email{shanlin\_huang@hust.edu.cn}
\address{Jiaqi Yu,  School of Mathematics and Statistics, Huazhong University of Science and Technology,  Wuhan,  430074,  P.R. China}
\email{jiaqi\_yu@hust.edu.cn}


\begin{document}
	
	\begin{abstract}
		Given a curve $\Gamma$ and a set $\Lambda$ in the plane, the concept of the Heisenberg uniqueness pair $(\Gamma, \Lambda)$ was first introduced by Hedenmalm and Motes-Rodr\'{\i}gez (Ann. of Math. 173(2),1507-1527, 2011, \cite{HM}) as a variant of the uncertainty principle for the Fourier transform. The main results in \cite{HM} concern the hyperbola $\Gamma_{\epsilon}=\{(x_1, x_2)\in \R^2,\, x_1x_2=\epsilon\}$ ($0\ne\epsilon\in \R$) and lattice-crosses $\Lambda_{\alpha\beta}=(\alpha\mathbb{Z}\times \{0\})\cup(\{0\}\times \beta\mathbb{Z})$ ($\alpha, \beta>0$), where it's proved that $(\Gamma_{\epsilon}, \Lambda_{\alpha\beta})$ is a Heisenberg uniqueness pair if and only if $\alpha\beta\leq 1/|\epsilon|$.

In this paper, we aim to study the endpoint case (i.e., $\epsilon=0$ in $\Gamma_{\epsilon}$)  and investigate the following problem: what's the minimal amount of information required on $\Lambda$ (the zero set) to form a  Heisenberg uniqueness pair?  When $\Lambda$ is contained in the union of two curves in the plane, we give characterizations in terms of some dynamical system conditions. The situation is quite different in higher dimensions and we  obtain characterizations  in the case that $\Lambda$ is the union of two hyperplanes.
		
		
	\end{abstract}
	
	\maketitle
	

	\section{Introduction and main results}\label{section1}
	
	\subsection{Background and motivation}\label{1.1} Let $\mu$ be a finite complex-valued Borel measure on the plane $\R^2$. The Fourier transform of $\mu$ is defined by
	\begin{align}\label{equ1.1}
		\widehat {\mu} (\xi_1, \xi_2)=\int_{\R^2}e^{-\pi i(x_1\xi_1+x_2\xi_2)}\d\mu(x_1, x_2),\;\mbox{for}\; (\xi_1,\xi_2)\in \R^2.
	\end{align}
	For $\Gamma$ a finite union of smooth curves which are disjoint (except possibly for the endpoints), denote by $\mathcal{AC}(\Gamma)$ the set of finite complex-valued measures supported in $\Gamma$  and also absolutely continuous with respect to the arc length of the curve.
	In \cite{HM}, Hedenmalm and Motes-Rodr\'{\i}gez first introduce the following notion of Heisenberg uniqueness pairs (\emph{HUP}):
	\begin{definition}\label{def1}
		Let $\Gamma$ be  a finite union of smooth disjoint curves and $\Lambda\subset \R^2$. $(\Gamma ,\Lambda)$ is said to be a Heisenberg uniqueness pair if
		\begin{align}\label{equ1.2}
			\mu\in \mathcal{AC}(\Gamma)\,\, \text{and}\,\,\,\, \widehat{\mu}|_{\Lambda}=0\Longrightarrow \mu=0.
		\end{align}
	\end{definition}
	It's natural to compare it with the uncertainty principle in Fourier analysis, which roughly speaking describes the interplay between the size of $\mu$ and the size of $\hat{\mu}$. In particular, we mention the related annihilating pair of sets of positive measure and refer to \cite[Chapter 3]{HJ}.
	The notion of \emph{HUP}, on the other hand, balances the support of a measure and the zero set of $\hat{\mu}$.  More recently, motivated by \cite{HM}, as well as the interpolation formula obtained by Radchenko and Viazovska \cite{RV}, P.G. Ramos and Sousa studied the so-called Fourier uniqueness pairs in \cite{RS}.
	
	Using the fact that Fourier transform is translation and rotation invariant, it follows that for any $x_0, \xi_0\in  \R^2$,
	\begin{align}\label{equ1.2.1}
		(\Gamma, \Lambda)\,\, \mbox{is a \emph{HUP}} \Longleftrightarrow (\Gamma+x_0, \Lambda+\xi_0)\,\, \mbox{is a \emph{HUP}}, \tag{Inv-1}
	\end{align}
	and
	\begin{align}\label{equ1.2.2}
		(\Gamma, \Lambda)\,\, \mbox{is a \emph{HUP}} \Longleftrightarrow (T^{-1}(\Gamma), T^{*}(\Lambda))\,\, \mbox{is a \emph{HUP}},  \tag{Inv-2}
	\end{align}
	where $T: \R^2\rightarrow\R^2$ is an invertible linear transform and $T^{*}$ is the  adjoint operator.
	
	The main result in \cite[Theorem 1.2, Corollary 1.3]{HM} shows that when
	\begin{equation}\label{equ1.3.0}
		\Gamma_{\epsilon}=\{(x_1, x_2)\in \R^2,\, x_1x_2=\epsilon\},
	\end{equation}
	where $\epsilon\ne 0$ is real, and $\Lambda$ is the lattice cross
	\begin{equation}\label{equ1.2.0}
		\Lambda_{\alpha\beta}=(\alpha\mathbb{Z}\times \{0\})\cup(\{0\}\times \beta\mathbb{Z}),
	\end{equation}
	where $\alpha, \beta$ are positive numbers. Then $(\Gamma_{\epsilon}, \Lambda_{\alpha\beta})$ is a \emph{Heisenberg uniqueness pair} if and only if $\alpha\beta\leq \frac{1}{|\epsilon|}$.

Since the pioneering work \cite{HM}, the study of Heisenberg uniqueness pairs  has been investigated  systematically, which we briefly recall.
For $\Gamma$ being a circle, Lev \cite{Le} and Sj\"{o}lin \cite{Sj} independently proved that if $\Lambda$ is the union of certain straight lines, then it's a \emph{HUP}, see \cite{Vi} for generalizations in higher dimensions. Sj\"{o}lin \cite{Sj2} also studied the case where $\Gamma$ is a parabola and $\Lambda=\Lambda_1\cup\Lambda_1$, where  $\Lambda_1$ and $\Lambda_2$ are subsets of two different straight lines, see also in \cite{GR}. When $\Gamma$ corresponds to parallel lines or other algebraic curves,  we refer to \cite{Ba,GS,GS2} and references therein.
	It's worth mentioning that Jaming and Kellay  \cite{JK} transfer the question of whether ($\Gamma, l_1\cup l_2$) ($l_1$, $l_2$ are two distinct lines) is a \emph{HUP} to the study of a certain dynamical system on $\Gamma$ generated by two lines. Using this new technique, they are able to unify and extend the above examples.  This was further extended to higher dimensions by Gr\"{o}chenig and Jaming \cite{GJ}.

	One interesting aspect of the \emph{HUP} is that it's closely related to uniqueness properties of linear PDEs when $\Gamma$ is an algebraic curve. Indeed, let $\Gamma=\{x\in\R^n:\,P(x)=0\}$, where $P$ is a polynomial on $\R^n$. If $\mu$ is a finite measure supported on $\Gamma$, then
	\begin{equation*}
		P\left(\frac{i\partial_1}{\pi},\cdots, \frac{i\partial_n}{\pi}\right)\hat{\mu}(\xi)=\int_{\R^n}e^{-\pi i(x_1\xi_1+\cdot+x_n\xi_n)}P(x)d\mu(x)=0.
	\end{equation*}
	Using invariance property \eqref{equ1.2.2}, the  corresponding PDE of $\Gamma_{\epsilon}$ in \eqref{equ1.3.0} is the one-dimensional Klein-Gordon equation.
	From the perspective of unique continuation of PDEs, it seems natural to consider the following problem:
	\begin{center}\label{que1}
		given a  curve $\Gamma\subset\R^2$,   what's the minimum required information on $\Lambda$  to form a \emph{HUP}?
	\end{center}
	

	\noindent \textbf{Aim and Motivation} $\;$ The main goal of this paper is to investigate this problem for $\Gamma$ being the endpoint case of \eqref{equ1.3.0}, i.e,
	\begin{align}\label{equ1.3}
		\Gamma_0=\{(x_1, x_2)\in \R^2,\, x_1\cdot x_2=0\}.
	\end{align}
	The corresponding PDE expresses the one-dimensional wave equation. In this case, as is pointed out in \cite[p.1510]{HM},
	``it appears that the characterization of unique pairs $(\Gamma_0, \Lambda)$ may get quite complicated''.
	Indeed, the situation is very different from the the case $\epsilon\ne 0$ since the following necessary condition holds:
	\begin{align}\label{equ1.6}
		\pi_1(\Lambda) \,\,\,\mbox{and}\,\,\, \pi_2(\Lambda)\,\, \mbox{are dense in}\,\, \R.
	\end{align}
	Here and in what follows, we use $\pi_1(\Lambda)$ to denote the orthogonal projection of $\Lambda$ to the $\xi_1$-axis, $\pi_2(\Lambda)$ denote the orthogonal projection of $\Lambda$ to the $\xi_2$-axis.
	
	Our aim is to give characterizations of the  \emph{Heisenberg uniqueness pair} $(\Gamma_0, \Lambda)$.
	We shall mainly study the typical case that $\Lambda$ is a subset of two lines or smooth curves in the plane.

	\subsection{Characterizations for Heisenberg uniqueness pairs in $\R^2$}
	In order to state the result, we need the following definition from dynamical systems.
	\begin{definition}\label{wander}
		Let $\Phi$ be a continuous bijection from $\R\to\R$. An open set $S\subset\R$ is said to be a wandering set associate with $\Phi$, if
		\begin{align}\label{wan-def}
			\Phi^m(S)\bigcap S=\emptyset\,\,\, \text{holds for any}\,\,\, m\in\mathbb{Z}\setminus\{0\}.
		\end{align}
		Here and in what follows, we use the notation $\Phi^k=\Phi\circ\Phi\circ\dots\circ\Phi$ and $\Phi^{-k}=\Phi^{-1}\circ\Phi^{-1}\circ\dots\circ\Phi^{-1}$ for any $k=1,2,\ldots$, $\Phi(S)=\{\Phi(t):\,t\in S\}$.   In particular, if $S=\emptyset$, we say that it's wandering for any $\Phi$.
	\end{definition}
	Our definition is slightly different from \cite[p.173]{BS}, \cite{JK}, where $\Phi$ is not a bijection. It deserves to mention that in \cite{JK},  \emph{Heisenberg uniqueness pairs} were obtained with the help of wandering sets. If $S$ is non-wandering, then there exist some interval $I\subset S$ and some $m\in\mathbb{Z}\setminus\{0\}$ such that $\Phi^m(I)\subset S$.

	First we study  the following  special case:
	\begin{center}
		$\Lambda$ consists of a portion which parallels to the $\xi_1$-axis or $\xi_2$-axis.
	\end{center}

	More precisely, we have
	\begin{theorem}\label{thm-1}
		Let $\Gamma_0$ be given by \eqref{equ1.3}. Assume that  either
		\begin{align}\label{equ1.7}
				\Lambda_1\subset \R\times \{a\},\quad \Lambda_2\subset \{b\}\times \R,
		\end{align}
		or
		\begin{align}\label{equ1.8}
				\Lambda_1\subset \R\times \{a\},\quad\Lambda_2\subset \left\{(\xi, T\xi),\,\,\xi\in\R\right\},
		\end{align}
		where $T:\,\,\R\to\R$ is a continuous bijection and $a, b\in \R$ are arbitrarily fixed.
		Then the following  statements are equivalent:
		
		\noindent $(i)$   $(\Gamma_0 ,\Lambda)$ is a Heisenberg uniqueness pair.

		\noindent $(ii)$   $\pi_1(\Lambda_1)$ and $\pi_2(\Lambda_2)$ are dense in $\R$.
	\end{theorem}
	
	Next, we consider the following general situation:
	\begin{align}\label{equ1.9}
		\Lambda=\Lambda_1\cup\Lambda_2,\,\,\,\, \Lambda_i\subset \{(\xi, T_i\xi),\,\,\xi\in\R\},  \,\,\,\,\, i=1,2,
	\end{align}
	where $T_1$ and $T_2$ are continuous bijection from $\R\to\R$.

	\begin{theorem}\label{thm-2}
		Let $\Gamma_0$ and $\Lambda$  be given in \eqref{equ1.3} and \eqref{equ1.9} respectively.
		
		\noindent $(\mathbf{A})$ Assume that  there exists an open interval $I\subset\R$ such that
		\begin{align}\label{equ1.10}
			\Phi^2(\xi)=\xi,\,\,\mbox{for all}\,\, \xi\in I,\,\,\,\, \mbox{where} \,\,\,\Phi=T_2^{-1}\circ T_1.
		\end{align}
		
		Then $(\Gamma_0,\Lambda)$ is \textbf{not} a Heisenberg uniqueness pair.
		
		\noindent $(\mathbf{B})$ Assume that  \eqref{equ1.10} doesn't hold for any open interval $I\subset\R$. Then the following  statements are equivalent:

		\noindent $(i)$   $(\Gamma_0 ,\Lambda)$ is a Heisenberg uniqueness pair.

		\noindent $(ii)$ $\pi_1(\Lambda)$ is dense in $\R$ and $\R\setminus(\overline{\pi_1(\Lambda_1)}\cap\overline{\pi_1(\Lambda_2)})$ is a wandering set associate with $\Phi=T_2^{-1}\circ T_1$.
	\end{theorem}
	We make the following remarks related to Theorem \ref{thm-1} and \ref{thm-2}:
	\begin{itemize}
		\item [($\textbf{a}_1$)] As far as we are aware, Theorem \ref{thm-1} and \ref{thm-2} are the first work to characterize the \emph{Heisenberg uniqueness pair} for the one dimensional wave equation. Besides the necessary condition \eqref{equ1.6}, the following sufficient condition was given in \cite{GJ,JK}: If $\Lambda=l_{\theta_1}\cup l_{\theta_2}$ and $\theta_1+\theta_2\ne \pi$, where $l_{\theta}=\{(t\cos\theta, t\sin\theta):\,t\in\R\}$, $0\leq \theta<\pi$, then $(\Gamma ,\Lambda)$ is a \emph{HUP}. It's direct to check that this is a special case of Theorem \ref{thm-2}-$(\mathbf{B})$, further, when $\theta_1+\theta_2=\pi$, we have $\Phi^2=I$, the identity operator, thus it's not a \emph{HUP} by Theorem \ref{thm-2}-$(\mathbf{A})$.

		\item [($\textbf{a}_2$)] In view of Theorem  \ref{thm-1} and the above result for  the Klein-Gordan equation  ($(\Gamma_{\epsilon}, \Lambda_{\alpha\beta})$ given by \eqref{equ1.3.0} and \eqref{equ1.2.0}), we see that when $\epsilon=0$, the situation for the wave equation is completely different. Moreover, in this case, the necessary condition \eqref{equ1.6} is also sufficient. On the other hand, Theorem  \ref{thm-2} shows that this is far from sufficient in general. In addition to the density assumption, it is essential that $\R\setminus(\overline{\pi_1(\Lambda_1)}\cap\overline{\pi_1(\Lambda_2)})$ satisfies the dynamical system condition associated with the map $\Phi$ (generated by $\Lambda_1$ and $\Lambda_2$). In Example \ref{ex-1}-\ref{ex-2}, we present concrete examples to illustrate this.
		\item [($\textbf{a}_3$)] The proof of  Theorem  \ref{thm-2} starts by observing that the assumption $\widehat{\mu}|_{\Lambda}=0$ is equivalent to a system of equations with two unknowns, see  \eqref{equ2.12.0}.  Then the key ingredient  is to use the wandering property of $\left(\overline{\pi_1(\Lambda_1)}\cap\overline{\pi_1(\Lambda_2)}\right)^c$ to obtain an iterative scheme.  Based on this iteration and the assumption $\mu\in\mathcal{AC}(\Gamma)$, we can solve the equations  \eqref{equ2.12.0} uniquely.

	\end{itemize}
	
	\subsection{Heisenberg uniqueness pairs in higher dimensions}
	There is no difficulty to extend the  definition of Heisenberg uniqueness pairs to higher dimensions. However, the situation is very different.
	Indeed, let $Q$ be the quadratic form
	\begin{equation*}
		Q(x)=x_1^2+x_2^2+\cdots+x_{n-1}^2-x_n^2
	\end{equation*}
	with the associated bilinear form $B:\, \R^n\times\R^n\rightarrow\R$ such that $Q(x)=B(x,x)$. Consider the cone
	\begin{align}\label{equ3.8}
		S=\{x\in \R^n,\, Q(x)=0\},
	\end{align}
	as well as the hyperplane
	\begin{align}\label{equ3.9}
		H_u=\{x\in\R^n:\,\langle x,u\rangle=0\},\quad |u|=1.
	\end{align}
	When $n=2$, it follows from Theorem \ref{thm-1} and Theorem \ref{thm-2}-$(\mathbf{A})$ that $(S, H_u)$  is not a Heisenberg uniqueness pair for any $u\in \mathbb{S}^1$.
	However, when $n>2$, it was proved in \cite[Corollary 2.3]{GJ} that $(S, H_u)$ is a Heisenberg uniqueness pair if and only if $Q(u)=0$.
	
	In the following we consider the union $H_{u_1}\cup H_{u_2}$ such that $Q(u_i)\ne 0$, $i=1, 2$.
	In order to state the result, we fix $u_1, u_2\in\mathbb{S}^{n-1}$ satisfying $u_1\neq \pm u_2$, then there exists a unique pair ($v_1, v_2$) (see Figure \ref{fig5}), such that $v_2\in \mathbb{S}^{n-1}\cap\{u_1\}^{\perp}$,  $v_1\in \mathbb{S}^{n-1}\cap\{u_2\}^{\perp}$ and
	\begin{equation}\label{equ1.20}
		u_2=\cos\theta_0u_1+\sin\theta_0v_2, \quad\,  v_1=\cos\theta_0v_2-\sin\theta_0u_1, \quad \theta_0\in(0,\pi).
	\end{equation}
	\begin{figure}\label{fig5}
		\begin{tikzpicture}[scale=0.90]
			\draw[->](0,0)--(2.3,0) node[right]{$u_1$};
			\draw[->](0,0)--(0,2.3) node[above]{$v_2$};
			\draw[->](0,0)--(2,1) node[above]{$u_2$};
			\draw[->](0,0)--(-1,2) node[above]{$v_1$} ;
			\coordinate (o) at (0,0) node[above,outer sep=-2.5pt,font=\small]at(0.9,0){$\theta_0$};
			\draw[thin,dotted](0,0)--(0,1.0) node[left,color=black,outer sep=-2.5pt,font=\small]{$\theta_0$};
			\coordinate (a) at (27:3);
			\coordinate (b) at (3:0);
			\coordinate (c) at (0:5);
			\coordinate (d) at (117:3);
			\draw[thin,dotted](2,1)--(2,0);
			\draw[thin,dotted](-1,2)--(0,2);
		\end{tikzpicture}
		\caption{$v_1,v_2$}\label{fig5}
	\end{figure}
	Further,  we set
	\begin{equation}\label{equ3.8.1}
		A:=Q(u_1),\,B:=B(u_1,v_2),\,C:=Q(v_2),\,D:=B(x,u_1),\,E:=B(x,v_2).
	\end{equation}
	
	%
	\begin{theorem}\label{thm-3}
		Let $S$ and $H_u$ be given in \eqref{equ3.8} and \eqref{equ3.9} respectively.
		Suppose $u_1,u_2\in \mathbb{S}^{n-1}$ are different unit vectors  satisfying $Q(u_1),\,Q(u_2)\neq 0$,  Let $A,\,B,\,C,\,\nu_1,\,\nu_2,\,\theta_0$ be  given by \eqref{equ3.8.1} and \eqref{equ1.20} respectively.
		
		\noindent $(i)$  	when $AC-B^2=0$, $(S\,,\,H_{u_1}\cup H_{u_2})$ is a Heisenberg uniqueness pair;
		
		\noindent $(ii)$  	when $AC-B^2<0$, $(S\,,\,H_{u_1}\cup H_{u_2})$ is a Heisenberg uniqueness pair if and only if
		\begin{equation}\label{equ1.12}
			B(v_1, v_2)\neq 0.
		\end{equation}
		
		\noindent $(iii)$  	when $AC-B^2>0$, $(S\,,\,H_{u_1}\cup H_{u_2})$ is a Heisenberg uniqueness pair if and only if
		\begin{equation}\label{equ1.13}
			\arctan\left(\frac{b}{a}\tan(\theta_0-\varphi)\right)+\arctan\left(\frac{b}{a}\tan\varphi\right)
			\notin\pi\mathbb{Q},
		\end{equation}
		where  $\frac{b}{a}=\sqrt{\frac{A+2B\tan\varphi+C\tan^2\varphi}{C-2B\tan\varphi+A\tan^2\varphi}}$ and  $\tan2\varphi=\frac{2B}{A-C}$. 
	\end{theorem}
	
	Some remarks related to Theorem \ref{thm-3} are as follows:
	\begin{itemize}
		\item [($\textbf{b}_1$)] Theorem \ref{thm-3} can be viewed as a complement of the result in \cite[Theorem 1.2]{GJ}, where it was proved that for general quadratic hypersurface $S$, there exists an exceptional set  $\mathcal{E}$ of measure zero in $\mathbb{S}^{n-1}\times \mathbb{S}^{n-1}$ such that $(S,\,\,H_{u_1}\cup H_{u_2})$ is a Heisenberg uniqueness pair provided that  $Q(u_1),\,Q(u_2)\neq 0$ and $u_1, u_2\notin \mathcal{E}$. Here we characterize explicitly the exceptional set when $S$ is the cone.
		
		\item [($\textbf{b}_2$)] We follow the strategy in \cite{GJ} and reduce matters to quadratic curves (the parabola, hyperbola and ellipse), thus the result has three different forms. It's worth mentioning that if we consider a subset $\Lambda\subset H_{u_1}\cup H_{u_2}$ instead, it seems that both the methods in proving Theorem \ref{thm-2} and  techniques in \cite{GJ,JK} are not enough,  and currently we don't know whether a dynamical system condition like \eqref{wan-def} will work in higher dimensions,  see Remark \ref{rmk3.2}.
		
	\end{itemize}

	\subsection{Organization of notation}
	The rest of the paper is organized as follows:
	Section \ref{section2} gives the proof of Theorem \ref{thm-1} and \ref{thm-2}, as well as some examples and further extensions. In Section \ref{section3},  the proof of   Theorem \ref{thm-3} is given.

	\section{Two dimensional Heisenberg Uniqueness Pairs}\label{section2}
	
	Let $\Gamma_0$ be given by \eqref{equ1.3} and $\mu$ be a measure absolutely continuous with respect to arc length on $\Gamma_0$. Then there exist $f, g\in L^1(\R)$ such that
	\begin{equation}\label{equ2.1}
		\d\mu(x_1, x_2)=f(x_1)\d x_1\d\delta_0(x_2)+g(x_2)\d\delta_0(x_1)\d x_2,
	\end{equation}
	thus, we have
	\begin{equation}\label{equ2.2}
		\widehat {\mu} (\xi_1, \xi_2)=\widehat{f}(\xi_1)+\widehat{g}(\xi_2),\;\,\,\mbox{for all}\:(\xi_1,\xi_2)\in \R^2.
	\end{equation}
	
	\subsection{Proof of Theorem \ref{thm-1}}
	We first consider the case that $\Lambda$ satisfies the assumption \eqref{equ1.7}. By invariance property \eqref{equ1.2.1}, we can assume that $a=b=0$, i.e.,
	$$
	\Lambda_1\subset \R\times \{0\},\quad\Lambda_2\subset \{0\}\times \R.
	$$
	In view of \eqref{equ2.2}, we have
	\begin{equation}\label{equ2.3}
		\widehat{\mu}|_{\Lambda}=0 \iff \left\{
		\begin{array}{ll}
			\widehat{f}(\xi_1)+\widehat{g}(0)=0,\quad \,\,\forall \xi_1\in\pi_1(\Lambda_1),\\[0.2cm]
			\widehat{f}(0)+\widehat{g}(\xi_2)=0,\quad \,\,\forall \xi_2\in\pi_2(\Lambda_2).
		\end{array}
		\right.
	\end{equation}
	
\noindent $(ii)\Rightarrow (i)$.	Assume that both $\pi_1(\Lambda_1)$ and $\pi_2(\Lambda_2)$ are dense in $\R$. Note that  $\hat{f}, \hat{g}\in C_0(\R)$, the space of continuous functions which tending to zero at $\infty$. Letting $\pi_1(\Lambda_1)\ni \xi_1\rightarrow \infty$ and $\pi_2(\Lambda_2)\ni \xi_2\rightarrow \infty$ we obtain from \eqref{equ2.3} that
	\begin{equation}\label{equ2.4}
		\hat{f}(0)=\hat{g}(0)=0.
	\end{equation}
	Combing \eqref{equ2.3} and \eqref{equ2.4} we have in turn that
	\begin{equation}\label{equ2.5}
		\hat{f}(\xi_1)=0,\,\,\,\,\,\,\forall\,\, \xi_1\in\pi_1(\Lambda_1);\,\,\,\,\hat{g}(\xi_2)=0,\,\,\,\,\,\,\forall\,\,\xi_2\in\pi_2(\Lambda_2).
	\end{equation}
	Therefore $\hat{f}$ and $\hat{g}=0$ vanish identicaly by the density of $\pi_1(\Lambda_1)$ and $\pi_2(\Lambda_2)$ respectively. By uniqueness of the Fourier transform we conclude that $\mu=0$.
	
\noindent $(i)\Rightarrow (ii)$.	Conversely, assume that $(\Gamma_0, \Lambda)$ is a \emph{HUP}, now we prove the density by contradiction. Indeed, if $\pi_1(\Lambda_1)$ is not dense in $\R$, then there exists an open interval $I_1\subset\R\setminus\pi_1(\Lambda_1)$. Now we choose a non-zero function $f$ such that
	\begin{align}\label{equ2.8}
		f\in L^1(\R),\,\,\,\,\,\mbox{supp}\,\hat{f}\subset I_1,\,\,\,\mbox{and}\,\,\,\int_{\R}{f(x)\,\d x}=0.
	\end{align}
	Thus we have $\hat{f}(0)=0$. By taking $g\equiv0$ and our construction of $f$ above, it follows from \eqref{equ2.3} that $\widehat{\mu}|_{\Lambda}=0$ and $\mu\neq0$, which is a contradiction. Therefore $\pi_1(\Lambda_1)$ is dense in $\R$. The proof of the density of $\pi_2(\Lambda_2)$ is the same.
	
	Next, We consider the case that $\Lambda$ satisfies  the assumption \eqref{equ1.8}. The proof only needs a slight modification. By invariance property \eqref{equ1.2.1} again, we can choose $a=0$ and $T0=0$. In view of \eqref{equ2.2}, we have
	\begin{equation}\label{equ2.6}
		\widehat{\mu}|_{\Lambda}=0 \iff \left\{
		\begin{array}{ll}
			\widehat{f}(\xi)+\widehat{g}(0)=0,\quad\;\;\;\,\,\,\mbox{for all}\,\,  \xi\in\pi_1(\Lambda_1),\\[0.2cm]
			\widehat{f}(\xi)+\widehat{g}(T\xi)=0,\quad\,\,\,\,\mbox{for all}\,\,  \xi\in\pi_1(\Lambda_2).
		\end{array}
		\right.
	\end{equation}
\noindent $(ii)\Rightarrow (i)$.	Assume that both $\pi_1(\Lambda_1)$ and $\pi_2(\Lambda_2)$ are dense in $\R$.  Letting $\pi_1(\Lambda_1)\ni \xi_1\rightarrow \infty$ and by the density of $\pi_1(\Lambda_1)$  we obtain from \eqref{equ2.6} that $\hat{g}(0)=0$ and $\hat{f}\equiv0$, which in turn implies that
	\begin{equation}\label{equ2.7}
	\hat{g}(T\xi)=0,\,\,\,\,\,\,\forall\,\,\xi\in\pi_1(\Lambda_2).
	\end{equation}
	Note that $\{T\xi,\,\,\xi\in \pi_1(\Lambda_2)\}=\pi_2(\Lambda_2)$. Then by the density of $\pi_2(\Lambda_2)$ and \eqref{equ2.7}, we obtain that $\hat{g}\equiv 0$. Therefore $\mu\equiv 0$.
	
\noindent $(i)\Rightarrow (ii)$.	Conversely, assume that $(\Gamma_0, \Lambda)$ is a \emph{HUP}. If $\pi_1(\Lambda_1)$ is not dense in $\R$, then there exists an open interval $I_2\subset\R\setminus\pi_1(\Lambda_1)$ satisfying $0\notin I_2$. Now we choose a non-zero function $h(\xi)\in C_0^\infty(I_2)$ and let $s(\xi)=-h(T^{-1}\xi)$.
	In addition, we denote by
	\begin{align}\label{equ2.9}
		f(x):=\mathcal{F}^{-1}(h)(x),\,\,\,\,g(x):=\mathcal{F}^{-1}(s)(x).
	\end{align}
	This, together with the definitions of $h$ and $s$, implies
	\begin{align}\label{equ2.10}
		\hat{f}(\xi)+\hat{g}(T\xi)=0,\,\,\,\mbox{for all}\,\, \xi\in\R.
	\end{align}
	In particular, taking $\xi=0$ in \eqref{equ2.10} and noticing that  $\hat{f}(0)=h(0)=0$ (since $0\notin \text{supp}h$), we obtain that $\hat{g}(0)=\hat{g}(T0)=0$. Therefore we also have
	\begin{align}\label{equ2.11}
		\hat{f}(\xi)+\hat{g}(0)=0,\quad\;\;\;\mbox{for all}\,\,  \xi\in\pi_1(\Lambda_1).
	\end{align}
	It follows from \eqref{equ2.6}, \eqref{equ2.10} and \eqref{equ2.11} that there exists some $\mu\ne 0$ but $\hat{\mu}|_{\Lambda}=0$. If $\pi_2(\Lambda_2)$ is not dense in $\R$, then there exists an open interval $I_3\subset\R\setminus\pi_1(\Lambda_2)$.  In this case, we construct the counterexample by taking $f\equiv 0$ and choosing some $0\ne g\in L^1(\R)$ such that
 $\mbox{supp}\,\hat{g}\subset T(I_3)$ and $\int_{\R}{g(x)\,\d x}=0$. Thus we have $\hat{g}(0)=0$ and $\hat{g}(T\xi)=0$, when $\xi\in \pi_1(\Lambda_2)$. Clearly, this shows that the right hand side of \eqref{equ2.6} holds and the proof is completed.


	\subsection{Proof of Theorem \ref{thm-2}, Part $(\mathbf{A})$ }
	
	First of all, by \eqref{equ2.2} we have
	
	\begin{align}\label{equ2.12.0}
		\hat{\mu}|_{\Lambda}=0 \iff \left\{
		\begin{array}{ll}
			\hat{f}(\xi)+\hat{g}(T_1\xi)=0,\quad \,\,\forall \xi\in\pi_1(\Lambda_1)\\[0.3cm]
			\hat{f}(\xi)+\hat{g}(T_2\xi)=0,\quad \,\,\forall \xi\in\pi_1(\Lambda_2)
		\end{array}
		\right.
	\end{align}
	Suppose $I\subset\R$ with $|I|<\infty$, such that \eqref{equ1.10} holds. Since $\Phi$ is a continuous bijection from $\R\to\R$, thus  $\Phi$ is strictly monotone. We divide the proof into two cases.
	
	\noindent \emph{Case 1. $\Phi$ is strictly increasing.} In this case, we claim that
	\begin{align}\label{equ2.12}
		\Phi(\xi)=\xi,\quad\;\;\;\mbox{for all}\,\, \xi\in I.
	\end{align}
	Indeed,  assume that there exists some $\xi_0\in I$ such that $\Phi(\xi_0)\neq\xi_0$. If $\Phi(\xi_0)<\xi_0$, by \eqref{equ1.10} we have
	$
	\xi_0=\Phi^2(\xi_0)<\Phi(\xi_0)<\xi_0,
	$
	which is a contradiction;
	if $\Phi(\xi_0)>\xi_0$, we have
	$
	\xi_0=\Phi^2(\xi_0)>\Phi(\xi_0)>\xi_0,
	$
	which is also a contradiction.   Thus the claim holds.
	
	To proceed, we take two non-trivial functions $h$ and $s$ such that
	\begin{align}\label{equ2.13}
		0\neq h(\xi)\in C_0^\infty(I),\,\,\,\,s(\xi)=-h(T_2^{-1}\xi)\in C_0^\infty(\R).
	\end{align}
	Now we denote by
	\begin{align}\label{equ2.14}
		f(x):=\mathcal{F}^{-1}(h)(x),\,\,\,\,g(x):=\mathcal{F}^{-1}(s)(x).
	\end{align}
	Then we have
	\begin{align}\label{equ2.15}
		\hat f(\xi)+\hat g(T_1\xi)=h(\xi)+s(T_1\xi)=h(\xi)-h(\Phi(\xi))=0, \;\;\;\mbox{for all}\,\, \xi\in\R,
	\end{align}
	where the first and the second equality follow from \eqref{equ2.14} and \eqref{equ2.13} respectively. The third equality follows from the fact that  when $\xi\in I$, $\Phi(\xi)=\xi$; while when  $\xi\in\R\setminus I$, then $\Phi(\xi)\in\R\setminus\Phi(I)=\R\setminus I$, which implies that $h(\xi)=h(\Phi(\xi))=0$.
	
	Moreover, we have
	\begin{align}\label{equ2.16}
		\hat f(\xi)+\hat g(T_2\xi)=h(\xi)+s(T_2\xi)=h(\xi)-h(\xi)=0, \;\;\;\mbox{for all}\,\, \xi\in\R.
	\end{align}
	By \eqref{equ2.12.0}, \eqref{equ2.15} and \eqref{equ2.16}, we conclude that $\mu\neq0$ but $\hat\mu|_{\Lambda}=0$. Therefore  ($\Gamma_0, \Lambda$) is not a Heisenberg uniqueness pair.
	
	\noindent \emph{Case 2. $\Phi$ is strictly  decreasing.}  In this case, we first claim there exists a unique $\xi_0\in\R$ such that
	$\Phi(\xi)=\xi$. Indeed, since $\Phi$ is monotonically decreasing as well as bijective, the existence follows from the  intermediate value theorem. Moreover, if there exist $\xi_1\ne \xi_2$ such that $\Phi(\xi_j)=\xi_j$, $j=1, 2$. Then we deduce that
	$$
	0<(\xi_1-\xi_2)(\xi_1-\xi_2)=(\xi_1-\xi_2)(\Phi(\xi_1)-\Phi(\xi_2))<0,
	$$
	which is a contradiction. Now we can choose an interval $\tilde{I}:=(a, b)$ such that
	\begin{align}\label{equ2.17}
		\tilde{I} \subset I,\,\,\,\,\,\xi_0\notin \tilde{I}.
	\end{align}
	Secondly, we claim that
	\begin{align}\label{equ2.18}
		\tilde{I}\cap\Phi(\tilde{I})=\emptyset,\,\,\,\,\,T_1(\tilde{I})\cap T_1\circ\Phi(\tilde{I})=\emptyset.
	\end{align}
	In fact, since $\Phi$ is monotonically decreasing, we have $\Phi(\tilde{I})= (\Phi(b),\, \Phi(a))$. When $\xi_0<a$, we have $\Phi(a)<\Phi(\xi_0)=\xi_0<a$; while when $\xi_0>b$, we have $\Phi(b)>\Phi(\xi_0)=\xi_0>b$. Thus we have $\tilde{I}\cap\Phi(\tilde{I})=\emptyset$. This, together with the fact that  $T_1$ is continuous bijection, yields  the second statement in \eqref{equ2.18}.
	
	Third, we  fix some  $0\neq h_1(\xi)\in C_0^\infty(\tilde{I})$. By \eqref{equ2.18}, we are allowed to define
	\begin{align*}
		h(\xi):=
		\begin{cases}
			h_1(\xi),\quad\quad\,\,\,\,\,\,\,\,\,\mbox{if}\,\,\xi\in \tilde{I},\\[0.1cm]
			h_1(\Phi^{-1}(\xi)),\,\,\,\,\,\,\,\mbox{if}\,\,\xi\in\Phi(\tilde{I}),\\[0.1cm]
			0,\,\,\qquad\qquad\,\, \,\,\mbox{else},
		\end{cases}
		s(\eta):=
		\begin{cases}
			-h_1(T_1^{-1}(\eta)),\qquad\,\,\,\,\,\,\,\mbox{if}\,\,\eta\in T_1(\tilde{I}),\\[0.1cm]
			-h_1(\Phi^{-1}\circ T_1^{-1}(\eta)),\,\,\,\,\mbox{if}\,\,\eta\in T_1\circ\Phi(\tilde{I}),\\[0.1cm]
			0,\,\,\,\,\qquad\qquad \qquad \quad \mbox{else},
		\end{cases}
	\end{align*}
	and
	\begin{align}\label{equ2.19}
		f:=\mathcal{F}^{-1}h,\,\,\,\,g:=\mathcal{F}^{-1}s.
	\end{align}
	%
	A direct computation shows
	\begin{align*}
		\widehat g(T_1\xi)&=
		\begin{cases}
			-h_1(T_1^{-1}(T_1\xi)),\qquad\quad\,\,\,\,\,\,\,\mbox{if}\,\,T_1\xi\in T_1(\tilde{I}),\\[0.1cm]
			-h_1(\Phi^{-1}\circ T_1^{-1}(T_1\xi)),\quad\,\,\,\,\mbox{if}\,\,\,T_1\xi\in T_1\circ\Phi(\tilde{I}),\\[0.1cm]
			0,\qquad\qquad\qquad\quad\quad\,\,\,\,\,\,\,\,\,\,\mbox{if}\,\, T_1\xi\in(T_1(\tilde{I})\cup T_1\circ\Phi(\tilde{I}))^c.
		\end{cases}\\
		&=\begin{cases}
			-h_1(\xi),\qquad\quad\qquad\,\,\,\,\,\,\,\,\,\,\,\,\,\,\,\mbox{if}\,\,\xi\in \tilde{I},\\[0.1cm]
			-h_1(\Phi^{-1}(\xi)),\quad\qquad\,\,\,\,\,\,\,\,\,\,\,\,\,\mbox{if}\,\,\xi\in\Phi(\tilde{I}),\\[0.1cm]
			0,\qquad\qquad\quad\qquad\,\,\,\,\,\,\,\,\,\,\,\,\,\,\,\,\mbox{if}\,\,\xi\in(\tilde{I}\cup\Phi(\tilde{I}))^c.
		\end{cases}
	\end{align*}
	Similarly, we have
	\begin{align*}
		\widehat g(T_2\xi)
		=\begin{cases}
			-h_1(\Phi^{-1}\circ\Phi^{-1}(\xi)),\quad\,\,\,\,\,\,\,\,\mbox{if}\,\,\xi\in \tilde{I},\\[0.1cm]
			-h_1(\Phi^{-1}(\xi)),\qquad\quad\,\,\,\,\,\,\,\,\,\,\,\mbox{if}\,\,\xi\in\Phi(\tilde{I}),\\[0.1cm]
			0,\qquad\qquad\qquad\quad\,\,\,\,\,\,\,\,\,\,\,\,\,\,\mbox{if}\,\, \xi\in(\tilde{I}\cup\Phi(\tilde{I}))^c.
		\end{cases}
	\end{align*}
	Using the fact that $\Phi^{2}(\xi)=\xi$ for all $\xi\in \tilde{I}$, we have
	\begin{align}\label{equ2.20}
		\hat f(\xi)+\hat g(T_1\xi)=0,\,\,\,\,\mbox{and}\,\,\,\,\hat f(\xi)+\hat g(T_2\xi)=0, \qquad \mbox{for all}\,\,\xi\in\R.
	\end{align}
	
It follows from \eqref{equ2.12.0}, \eqref{equ2.20} that there exists $\mu\neq0$, but  $\hat\mu|_{\Lambda}=0$. Therefore ($\Gamma_0,\, \Lambda$) is not a Heisenberg uniqueness pair.
	
	\subsection{Proof of Theorem \ref{thm-2}, Part $(\mathbf{B})$}
	We organize the proof by two steps.
	
	\noindent{\it Step 1. We show that (ii) of Theorem \ref{thm-2}, Part $(\mathbf{B})$ implies
		(i).}
	
We start by rewriting  the  the equivalence relation \eqref{equ2.12.0} as follows
	\begin{subnumcases}
		{\hat{\mu}|_{\Lambda}=0 \iff }
		\hat{f}(\xi)+\hat{g}(T_1\xi)=0,\,\,\,\,\,\mbox{for all}\,\,\xi\in\pi_1(\Lambda_1),\label{equ2.22a}\\[0.1cm]
		\hat{f}(\eta)+\hat{g}(T_2\eta)=0,\,\,\,\,\,\mbox{for all}\,\, \eta\in\pi_1(\Lambda_2).\label{equ2.22b}
	\end{subnumcases}
	\begin{subnumcases}
		{\quad\qquad\iff}
		\hat{f}(T_1^{-1}\xi)+\hat{g}(\xi)=0,\,\,\,\,\,\mbox{for all}\,\, \xi\in\pi_2(\Lambda_1),\label{equ2.23a}\\[0.1cm]
		\hat{f}(T_2^{-1}\eta)+\hat{g}(\eta)=0,\,\,\,\,\,\mbox{for all}\,\,\eta\in\pi_2(\Lambda_2).\label{equ2.23b}
	\end{subnumcases}	
	We shall use these equations, together with the geometric condition (ii) in   Theorem \ref{thm-2} and the assumption that \eqref{equ1.10} doesn't hold, to solve $f$ and $g$ uniquely.
	Firstly, we need the following observation

\textbf{Claim 1.}  If  \eqref{equ1.10} doesn't hold for any open interval, then the fixed points of $\Phi^2$
	\begin{align}\label{equ2.22}
		J:=\{\xi\in\R:\Phi^2(\xi)=\xi\}
	\end{align}
	is nowhere dense and we have
	\begin{align}\label{equ2.23}
		\displaystyle\lim_{n\to\pm\infty}\Phi^{2n}(\xi)\in J\cup\{\pm\infty\},\,\,\,\,\mbox{for all}\,\,\xi\in\R.
	\end{align}
	%

\emph{Proof of Claim 1:}	If J is not nowhere dense,  one can find an open interval $I_0$ such that $J\cap I_0\subset I_0$ is dense, then by the continuity of $\Phi$ we obtain that $\Phi^2(\xi)=\xi$ for all $\xi\in I_0$, which is a contradiction. In order to prove \eqref{equ2.23}, notice that when
	$\xi\in J$, we have $\xi=\Phi^{2}(\xi)$, thus $\displaystyle\lim_{n\to\pm\infty}\Phi^{2n}(\xi)=\xi\in J$.
	On the other hand, for any
	$\xi\in\R\setminus J$, using the fact that $\{\Phi^{2n}(\xi)\}_{n\in\N}$ and $\{\Phi^{-2n}(\xi)\}_{n\in\N}$ are monotone sequences ($\Phi^2$ is strictly increasing),  it follows that either the limit of the sequence $\{\Phi^{2n}(\xi)\}_{n\in\N}$  exists or  it diverges to infinity. If
	\begin{align*}
		\displaystyle\lim_{n\to\pm\infty}\Phi^{2n}(\xi)=\xi_0\in \R,
	\end{align*}
	then it follows that
	\begin{align*}
		\xi_0=\displaystyle\lim_{n\to\infty}\Phi^{2n}(\xi)=\displaystyle\lim_{n\to\infty}\Phi^{2n+2}(\xi)=\Phi^2(\xi_0).
	\end{align*}
	Hence we deduce that $\xi_0\in J$, and the proof of \eqref{equ2.23} is complete.
	
	Next, let $E_j=\left(\overline{\pi_1(\Lambda_j)}\right)^c$ ($j=1, 2$), then $\pi_1(\Lambda_j)\cap E_j^c$ is dense in $E_j^c$. Denote by $E:=E_1\cup E_2$ and
	\begin{align}\label{equ2.26}
		F_j:=T_jE_j,\,\,\,j=1,\,2,\,\,\,\,\,F:= F_1\cup F_2.
	\end{align}
 Then $F_j=\left(\overline{\pi_2(\Lambda_j)}\right)^c$ ($j=1, 2$) and $\pi_2(\Lambda_j)\cap F_j^c$ is dense in $F_j^c$. We shall apply the non-wandering condition (ii) to deduce the following
%

\textbf{Claim 2.} Let $\hat{f}(J)=\{\hat{f}(\xi):\xi\in J\}$, where  $J$ is given in \eqref{equ2.22}. Then for any $\xi\in E$,

	\begin{align}\label{equ2.27}
		\hat{f}(\xi)\in\hat{f}(J)\cup\{0\},
	\end{align}
and for any $\eta\in F$,
	\begin{align}\label{equ2.28}
		\hat{g}(\eta)\in-\hat{f}(J)\cup\{0\}.
	\end{align}

\emph{Proof of Claim 2:}
	The key observation is  that the condition (ii) in Theorem \ref{thm-2} implies
	\begin{align}\label{equ2.29}
		\left\{
		\begin{array}{ll}
			\Phi^k(E_1)\cap\pi_1(\Lambda_1)\mbox{ is dense in }\Phi^k(E_1),\mbox{ for all }k\in\mathbb{Z}\setminus\{0\}.\\[0.05cm]
			\Phi^k(E_1)\cap\pi_1(\Lambda_2)\mbox{ is dense in }\Phi^k(E_1),\mbox{ for all }k\in\mathbb{Z},
		\end{array}
		\right.
	\end{align}
and
	\begin{align}\label{equ2.30}
		\left\{
		\begin{array}{ll}
			\Phi^k(E_2)\cap\pi_1(\Lambda_1)\mbox{ is dense in }\Phi^k(E_2),\mbox{ for all }k\in\mathbb{Z}.\\[0.05cm]
			\Phi^k(E_2)\cap\pi_1(\Lambda_2)\mbox{ is dense in }\Phi^k(E_2),\mbox{ for all }k\in\mathbb{Z}\setminus\{0\}.
		\end{array}
		\right.
	\end{align}
	In fact, since $\pi_1(\Lambda)$ is dense in $\R$, it follows from the definition of $E_1$ that $E_1\cap\pi_1(\Lambda_2)$ is dense in $E_1$, this proves the case $k=0$ in \eqref{equ2.29}; If there exists some $k_0\in\mathbb{Z}\setminus\{0\}$ such that $\Phi^{k_0}(E_1)\cap\pi_1(\Lambda_j)$ ($j=1$ or $2$) is not dense in $\Phi^{k_0}(E_1)$, then there exists some interval $I_0$ such that
	\begin{align*}
		I_0\subset\Phi^{k_0}(E_1)\,\,\,\mbox{and}\,\,\,I_0\subset\left(\overline{\pi_1(\Lambda_j)}\right)^c, \quad \,j=1\,\mbox{or}\,2.
	\end{align*}
	Denote by $\tilde{I}=\Phi^{-k_0}(I_0)$, then we have
	\begin{align*}
		\tilde{I}\subset E_1=(\overline{\pi_1(\Lambda_1)})^c\,\,\,\mbox{and}\,\,\,\Phi^{k_0}(\tilde{I})=I_0\subset\left(\overline{\pi_1(\Lambda_j)}\right)^c, \quad \,j=1\,\mbox{or}\,2,
	\end{align*}
	which contradicts with the assumption that $(\overline{\pi_1(\Lambda_1)}\cap\overline{\pi_1(\Lambda_2)})^c$  is a wandering set with $\Phi$. This finishes the proof of \eqref{equ2.29} and the proof of \eqref{equ2.30} is the same.
	
	Furthermore, recall that $T_j\pi_1(\Lambda_j)=\pi_2(\Lambda_j),\,j=1,2$, and for any $k\in\mathbb{Z}$,
	$$
T_1\Phi^k(E_1)=(T_1T_2^{-1})^kF_1,\,\,\,T_2\Phi^k(E_1)=(T_1T_2^{-1})^{k-1}F_1,
$$
$$T_1\Phi^{-k}(E_2)=(T_1T_2^{-1})^{-(k-1)}F_2,\,\,\,T_2\Phi^{-k}(E_2)=(T_1T_2^{-1})^{-k}F_2$$
As a result,  the following statements concerning $F_1,\,F_2$ follow from \eqref{equ2.29} and \eqref{equ2.30} respectively
	\begin{align}\label{equ2.31}
		\left\{
		\begin{array}{ll}
			(T_1T_2^{-1})^kF_1\cap\pi_2(\Lambda_1)\mbox{ is dense in }(T_1T_2^{-1})^kF_1,\mbox{ for all }k\in\mathbb{Z}\setminus\{0\}.\\[0.05cm]
			(T_1T_2^{-1})^kF_1\cap\pi_2(\Lambda_2)\mbox{ is dense in }(T_1T_2^{-1})^kF_1,\mbox{ for all }k\in\mathbb{Z},
		\end{array}
		\right.
	\end{align}
and
	\begin{align}\label{equ2.32}
		\left\{
		\begin{array}{ll}
			(T_1T_2^{-1})^kF_2\cap\pi_2(\Lambda_1)\mbox{ is dense in }(T_1T_2^{-1})^kF_2,\mbox{ for all }k\in\mathbb{Z}.\\[0.05cm]
			(T_1T_2^{-1})^kF_2\cap\pi_2(\Lambda_2)\mbox{ is dense in }(T_1T_2^{-1})^kF_2,\mbox{ for all }k\in\mathbb{Z}\setminus\{0\}.
		\end{array}
		\right.
	\end{align}
	
	Now we are ready to prove the claim. The observation  \eqref{equ2.29}, together with the continuity of $\hat{f}$ and $\hat{g}$,   allows us to extend
the domain of \eqref{equ2.22a}, \eqref{equ2.22b} to
	\begin{align}\label{equ2.33}
		\left\{
		\begin{array}{ll}
			\hat{f}(\xi)+\hat{g}(T_1\xi)=0,\,\,\,\,\,\,\,\mbox{for all}\,\, \xi\in\Phi^{-k-1}(E_1),\\[0.05cm]
			\hat{f}(\eta)+\hat{g}(T_2\eta)=0,\,\,\,\,\,\,\,\mbox{for all}\,\, \eta\in\Phi^{-k}(E_1),
		\end{array}
		\right.
	\end{align}
	where $k=0,1,2,\ldots$. For any $\eta\in\Phi^{-k}(E_1)$, we set $\xi=\Phi^{-1}\eta\in\Phi^{-k-1}(E_1)$ in \eqref{equ2.33}, then the above equations yield that for all $k\in \mathbb{N}$,
	\begin{align}\label{equ2.34}
		\hat f(\eta)=\hat f(\Phi^{-1}\eta)=\cdots=\hat f(\Phi^{-k}\eta),\quad\,\,\,\,\,\mbox{for all}\,\,\eta\in E_1.
	\end{align}
	Combining this with\textbf{ Claim 1}, we deduce that for all $\eta\in E_1$,
	\begin{align}\label{equ2.35}
		\hat f(\eta)=\displaystyle\lim_{n\to-\infty}\hat f(\Phi^{2n}(\eta))\in\{\hat f(\xi_0):\,\xi_0\in J\}\cup\{\hat f(\infty)\}.
	\end{align}
By Riemann-Lebesgue lemma, we have $\hat f(\infty)=0$, thus we prove \eqref{equ2.27} for $\xi\in E_1$. Using \eqref{equ2.30}, we obtain \eqref{equ2.27} for $\xi\in E_2$ with a minor modification.
	
	The proof of \eqref{equ2.28} is similar. Indeed,  we use \eqref{equ2.23a}, \eqref{equ2.23b} and the observation \eqref{equ2.31} to  obtain that  for any $k=0,1,2,\ldots$,
	\begin{align}\label{equ2.36}
		\left\{
		\begin{array}{ll}
			\hat{f}(T_1^{-1}\xi)+\hat{g}(\xi)=0,\,\,\,\,\,\,\,\mbox{for all}\,\, \xi\in(T_1T_2^{-1})^{k+1}F_1,\\[0.05cm]
			\hat{f}(T_2^{-1}\eta)+\hat{g}(\eta)=0,\,\,\,\,\,\,\,\mbox{for all}\,\, \eta\in(T_1T_2^{-1})^kF_1.
		\end{array}
		\right.
	\end{align}
	For any $\eta\in (T_1T_2^{-1})^kF_1$, let $\xi=(T_1T_2^{-1})\eta\in(T_1T_2^{-1})^{k+1}F_1$ in \eqref{equ2.36}, then the above equations yield that
	\begin{align}\label{equ2.37}
		\hat g(\eta)=\hat g((T_1T_2^{-1})\eta)=\cdots=\hat g((T_1T_2^{-1})^{k}\eta)=-\hat f(\Phi^kT_2^{-1}\eta)
	\end{align}
	holds for all $k\in \mathbb{N}$ and $\eta\in F_1$. Thus, similar to \eqref{equ2.35}, we have for all $\eta\in F_1$,
	\begin{align}\label{equ2.38}
		\hat g(\eta)=-\displaystyle\lim_{n\to+\infty}\hat f(\Phi^{2n}(T_2^{-1}\eta))\in\{-\hat f(\xi_0):\,\xi_0\in J\}\cup\{0\},
	\end{align}
which proves \eqref{equ2.28} for $\eta\in F_1$. The case $\eta\in F_2$ is the same except that we use  \eqref{equ2.32} instead of  \eqref{equ2.31}.  Therefore the proof of  Claim 2 is complete.

	Third,  we establish the following

\textbf{Claim 3.} \eqref{equ2.27} holds for all $\xi\in E^c$.


\emph{Proof of Claim 3:} There are two cases to consider.  (\romannumeral1)\,   $\Phi^k(\xi)\in E^c$  holds for all $k\in\N$.
Note that if $\xi,\Phi(\xi)\in E^c$, then $T_1\xi\in T_1(E^c)\subset F_1^c$ and $T_1\xi=T_2\Phi(\xi)\in T_2(E^c)\subset F_2^c$, that is $T_1\xi\in F^c$.
This indicates that $T_1\Phi^k(\xi)\in F^c$, where  $k\in\N$.  Since $\pi_1(\Lambda_1)$ is dense in $E^c$ and $\pi_2(\Lambda_2)$ is dense in $F^c$,
	we use \eqref{equ2.22a},  \eqref{equ2.23b} repeatedly and obtain that \eqref{equ2.35} holds;
	 (\romannumeral2)\,  there exists $k_1\in\N^\ast$ such that $\Phi^{k_1}(\xi)\in E$ and
	\begin{align}\label{equ2.41}
		\Phi^{j}(\xi)\in E^c\,\,\,\,\mbox{for}\,\,\,j=0,1,\dots,k_1-1.
	\end{align}
	Similar to case \noindent $(i)$,   we have
	\begin{align}\label{equ2.42}
		T_1\Phi^{j}(\xi)\in F^c\,\,\,\,\mbox{for}\,\,j=0,1,\dots,k_1-2 (k_1\ge2).
	\end{align}
	In view of \eqref{equ2.41} and \eqref{equ2.42}, we iterate  \eqref{equ2.22a} ($k_1$ times) and \eqref{equ2.23b} ($k_1-1$ times) to derive that
	\begin{align*}
		\hat f(\xi)=-\hat g(T_1\xi)=\hat f(\Phi(\xi))=\cdots=\hat f(\Phi^{k_1-1}(\xi))=-\hat g(T_1\Phi^{k_1-1}(\xi)).
	\end{align*}
	If $T_1\Phi^{k_1-1}(\xi)\in F^c$, since $\pi_2(\Lambda_2)$ is dense in $F^c$, we can choose $\eta=T_1\Phi^{k_1-1}(\xi)$ in \eqref{equ2.23b} and it follows that $-\hat g(T_1\Phi^{k_1-1}(\xi))=\hat f(\Phi^{k_1}(\xi))$. Note that $\Phi^{k_1}(\xi)\in E$ by assumption, then we apply \eqref{equ2.27} in \textbf{Claim 2} to obtain that
	\begin{align}\label{equ2.43}
		\hat f(\xi)=\hat f(\Phi^{k_1}(\xi))\in\{\hat{f}(\xi_0):\xi_0\in J\}\cup\{0\}.
	\end{align}
If $T_1\Phi^{k_1-1}(\xi)\in F$, we apply  \eqref{equ2.28} in \textbf{Claim 2} to deduce that
	\begin{align}\label{equ2.44}
		\hat f(\xi)=-\hat g(T_1\Phi^{k_1-1}(\xi))\in\{\hat{f}(\xi_0):\xi_0\in J\}\cup\{0\}.
	\end{align}
	Therefore, we finish the proof of the claim.

Finally, combining \textbf{Claim 3} with \eqref{equ2.27}   shows that
	\begin{align}\label{equ2.40}
		\hat f(\xi)\in\{\hat f(\xi_0):\,\xi_0\in J\}\cup\{0\},\,\forall\xi\in\R.
	\end{align}
On the other hand, by \textbf{Claim 1},  $J$ is nowhere dense. Since $\hat{f}\in C_0(\R)$, it follows that $\hat f(\xi)=0$ for all $\xi\in \R$.
Finally, we observe that \noindent $(ii)$ of Theorem \ref{thm-2}, Part $(\mathbf{B})$ also implies that  $\pi_2(\Lambda)$ is  dense in $\R$. Indeed, if $\pi_2(\Lambda)$ is not dense, then there exists some interval $I_0\subset(\overline{\pi_2(\Lambda)})^c=(\overline{\pi_2(\Lambda_1)})^c\bigcap(\overline{\pi_2(\Lambda_2)})^c$, and if we take $I:=T_1^{-1}(I_0)\subset (\overline{\pi_1(\Lambda_1)})^c$, then $\Phi(I)=T_2^{-1}(I_0)\subset (\overline{\pi_1(\Lambda_2)})^c$, this corresponds to the case 2 with $m=1$ in Proposition \ref{prop-wan} below and contradicts with the wandering property of $\Phi$. By \eqref{equ2.23a} and \eqref{equ2.23b} we deduce that $\hat g(\eta)=0$ for all $\eta\in \R$ . Therefore $f=g\equiv0$, i.e., $\mu=0$ and  $(\Gamma_0,\Lambda)$ is a \emph{Heisenberg uniqueness pair}.

		\noindent{\it Step 2. We show that (i) of Theorem \ref{thm-2}, Part $(\mathbf{B})$ implies
			(ii).}
		
		
		We prove by contradiction. First, if $\pi_1(\Lambda)$ is not dense in $\R$,  then there exists an open interval $I\subset(\pi_1(\Lambda))^c$. We construct counterexamples as follows:  Choose $f, g$  such that
		\begin{align*}
			0\neq f\in L^1(\R), \quad \mbox{supp} \hat f\subset I, \,\,\,\, g\equiv 0.
		\end{align*}
		Then one has $\mu\neq0$, while $\hat\mu|_\Lambda=0$. Thus $(\Gamma_0,\Lambda)$ is not a \emph{HUP} and the condition $\pi_1(\Lambda)$ is dense in $\R$ is necessary.
		
		Second, if $\left(\overline{\pi_1(\Lambda_1)}\cap\overline{\pi_1(\Lambda_2)}\right)^c$ is a non-wandering set associate with $\Phi=T_2^{-1}\circ T_1$, Then by the following Proposition \ref{prop-wan}, $(\Gamma_0,\Lambda)$ is not a \emph{HUP}. Therefore  the proof is complete.
		
		\begin{proposition}\label{prop-wan}
			Assume that  \eqref{equ1.10} doesn't hold for any open interval $I\subset\R$. If $(\overline{\pi_1(\Lambda_1)}\cap\overline{\pi_1(\Lambda_2)})^c$ is a non-wandering set associate with $\Phi=T_2^{-1}\circ T_1$, then $(\Gamma_0,\Lambda)$ is not a Heisenberg uniqueness pair.
		\end{proposition}
		\begin{proof}
			We first point out that $(\overline{\pi_1(\Lambda_1)})^c\cup(\overline{\pi_1(\Lambda_2)})^c$ is a non-wandering set associate with $\Phi$ if and only if there exists some interval $I\subset \R$ such that  at least one of the following four cases holds (in Figure \ref{fig2} below, we present related examples of non-wandering sets with $T_1\xi=\xi$, $T_2\xi=-2\xi$):
			
			\emph{Case 1:} $I\subset\left(\overline{\pi_1(\Lambda_1)}\right)^c$, and $\Phi^m(I)\subset\left(\overline{\pi_1(\Lambda_1)}\right)^c\,\,\,\mbox{for some}\,\,m\in\N^+$;
			
			\emph{Case 2:}  $I\subset\left(\overline{\pi_1(\Lambda_1)}\right)^c$, and  $\Phi^m(I)\subset\left(\overline{\pi_1(\Lambda_2)}\right)^c\,\,\,\mbox{for some}\,\,m\in\N^+$;
			
			\emph{Case 3:}  $I\subset\left(\overline{\pi_1(\Lambda_2)}\right)^c$, and  $\Phi^m(I)\subset\left(\overline{\pi_1(\Lambda_1)}\right)^c\,\,\,\mbox{for some}\,\,m\in\N^+$;
			
			\emph{Case 4:}  $I\subset\left(\overline{\pi_1(\Lambda_2)}\right)^c$, and  $\Phi^m(I)\subset\left(\overline{\pi_1(\Lambda_2)}\right)^c\,\,\,\mbox{for some}\,\,m\in\N^+$.

			\begin{figure}[ht]
				\centering
				\subfigure[\emph{Case 1}]{
					\begin{tikzpicture}[scale=0.90]
						\draw[->](-2,0)--(3,0) node[right]{$\xi$};
						\draw[->](0,-2)--(0,3) node[above]{$\eta$};
						\draw[blue](2,2)--(2.8,2.8) node[above]{$\Lambda_1$};
						\draw[blue](-0.5,-0.5)--(1,1);
						\draw[blue](-2,-2)--(-1,-1);
						\draw[red](1.2,-2.4)--(-1.1,2.2) node[above]{$\Lambda_2$};
						\draw[thin,dotted](-1,-1)--(-1,2);
						\draw[thin,dotted](-0.5,-0.5)--(-0.5,1);
						\draw[thin,dotted](-1,2)--(2,2);
						\draw[thin,dotted](-0.5,1)--(1,1);
						\draw[thin,dotted](1,0)--(1,1);
						\draw[thin,dotted](2,0)--(2,2);
						\draw[green](1,0)--(2,0) node[below left,color=black,scale=0.7]{I};
						\draw[green](0,1)--(0,2) node[below right,color=black,scale=0.7]{$T_1(I)$};
						\draw[green](-1,0)--(-0.5,0) node[below left,color=black,scale=0.6]{$\Phi(I)$};
					\end{tikzpicture}
				}
				\subfigure[\emph{Case 2}]{
					\begin{tikzpicture}[scale=0.90]
						\draw[->](-2,0)--(3,0) node[right]{$\xi$};
						\draw[->](0,-2)--(0,3) node[above]{$\eta$};
						\draw[blue](2,2)--(2.8,2.8) node[above]{$\Lambda_1$};
						\draw[red](1,-2)--(0.5,-1);
						\draw[blue](-1.5,-1.5)--(1,1);
						\draw[red](0.25,-0.5)--(-1.5,3) node[above]{$\Lambda_2$};
						\draw[thin,dotted](-0.5,-0.5)--(0.25,-0.5);
						\draw[thin,dotted](-1,-1)--(0.5,-1);
						\draw[thin,dotted](-1,-1)--(-1,2);
						\draw[thin,dotted](-0.5,-0.5)--(-0.5,1);
						\draw[thin,dotted](-1,2)--(2,2);
						\draw[thin,dotted](-0.5,1)--(1,1);
						\draw[thin,dotted](1,0)--(1,1);
						\draw[thin,dotted](2,0)--(2,2);
						\draw[thin,dotted](0.25,0)--(0.25,-0.5);
						\draw[thin,dotted](0.5,0)--(0.5,-1);
						\draw[green](1,0)--(2,0) node[below left,color=black,scale=0.6]{I};
						\draw[green](0,1)--(0,2) node[below right,color=black,scale=0.7]{$T_1(I)$};
						\draw[green](0.25,0)--(0.5,0) node[below,color=black,scale=0.5]{$\Phi^2(I)$};
					\end{tikzpicture}
				}
				\subfigure[\emph{Case 3}]{
					\begin{tikzpicture}[scale=0.90]
						\draw[->](-2,0)--(3,0) node[right]{$\xi$};
						\draw[->](0,-2.8)--(0,2.5) node[above]{$\eta$};
						\draw[blue](-0.25,-0.25)--(2,2) node[above]{$\Lambda_1$};
						\draw[blue](-1,-1)--(-0.5,-0.5);
						\draw[red](0.5,-1)--(-1,2) node[above]{$\Lambda_2$};
						\draw[red](1.3,-2.6)--(1,-2);
						\draw[thin,dotted](0.5,-1)--(0.5,0.5);
						\draw[thin,dotted](1,-2)--(1,1);
						\draw[thin,dotted](-0.25,0.5)--(0.5,0.5);
						\draw[thin,dotted](-0.5,1)--(1,1);
						\draw[thin,dotted](-0.5,-0.5)--(-0.5,1);
						\draw[thin,dotted](-0.25,-0.25)--(-0.25,0.5);
						\draw[green](0.5,0)--(1,0) node[below left,color=black,scale=0.6]{I};
						\draw[green](-0.5,0)--(-0.25,0) node[above left,color=black,scale=0.6]{$\Phi(I)$};
						\draw[green](0,0.5)--(0,1) node[below right,color=black,scale=0.5]{$T_1(I)$};
					\end{tikzpicture}
				}
				\subfigure[\emph{Case 4}]{
					\begin{tikzpicture}[scale=0.90]
						\draw[->](-1.8,0)--(2.3,0) node[right]{$\xi$};
						\draw[->](0,-2.3)--(0,2.3) node[above]{$\eta$};
						\draw[blue](-1.5,-1.5)--(2,2) node[above]{$\Lambda_1$};
						\draw[red](1.2,-2.4)--(1,-2);
						\draw[red](0.5,-1)--(-0.25,0.5);
						\draw[red](-0.5,1)--(-1,2) node[above]{$\Lambda_2$};
						\draw[thin,dotted](0.5,-1)--(0.5,0.5);
						\draw[thin,dotted](1,-2)--(1,1);
						\draw[thin,dotted](-0.25,0.5)--(0.5,0.5);
						\draw[thin,dotted](-0.5,1)--(1,1);
						\draw[thin,dotted](-0.25,0)--(-0.25,0.5);
						\draw[thin,dotted](-0.5,0)--(-0.5,1);
						\draw[green](0.5,0)--(1,0) node[below left,color=black,scale=0.6]{I};
						\draw[green](0,0.5)--(0,1) node[below right,color=black,scale=0.5]{$T_1(I)$};
						\draw[green](-0.5,0)--(-0.25,0) node[below left,color=black,scale=0.5]{$\Phi(I)$};
					\end{tikzpicture}
				}
				\caption{Non-wandering sets associate with $\Phi$}\label{fig2}
			\end{figure}
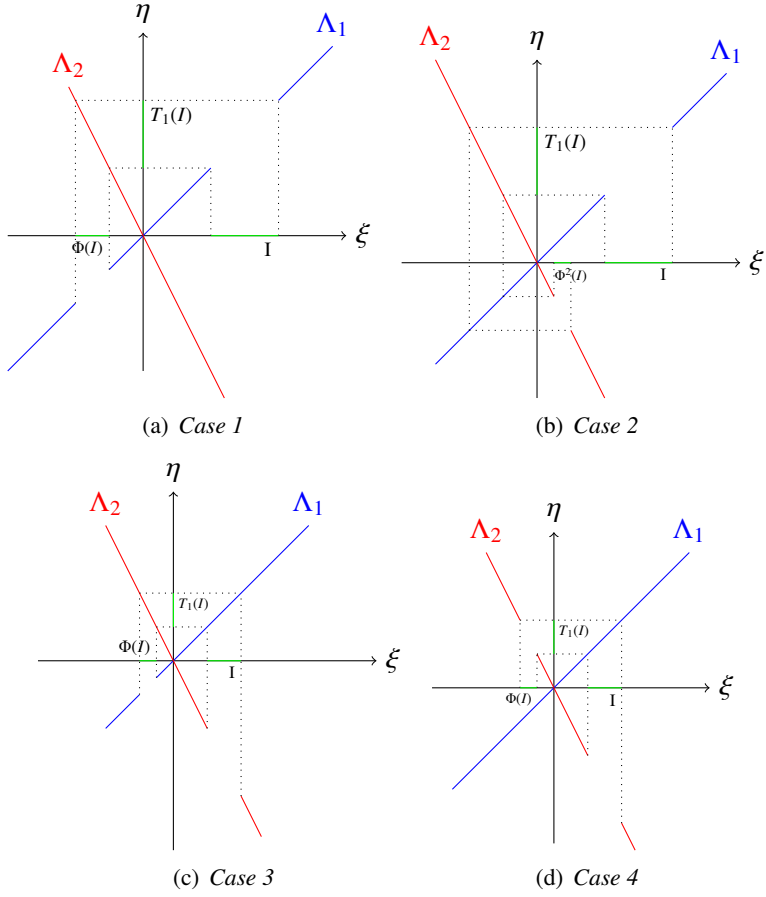
			We only provide counterexamples for \emph{case 1}, since the proof of other cases are the same. By assumption we have $I\ne \Phi^2(I)$, thus we can take a subinterval $I_1\subset I$, such that $I_1\cap\Phi^2(I_1)=\emptyset$. Moreover, we have $I_1\nsubseteq\Phi(I_1)$ (if $I_1\subseteq\Phi(I_1)$, then $I_1\subseteq\Phi(I_1)\subseteq\Phi^2(I_1)$, which is a contradiction), this shows that $I_1\cap\Phi(I_1)$ is strictly contained in $I_1$. Now we take another subinterval $I_0\subset I_1\setminus(I_1\cap\Phi(I_1))$ and the above procedure indicates that
\begin{align}\label{equ2.48}
				I_0\cap\Phi(I_0)=\emptyset\quad\, \mbox{and}\quad\, I_0\cap\Phi^2(I_0)=\emptyset.
			\end{align}
%
By \eqref{equ2.48}, we  can prove the following
			
			\textbf{Claim 4.}
			\begin{align}\label{equ2.49}
				\Phi^j(I_0)\cap\Phi^k(I_0)=\emptyset,\qquad \mbox{for all}\, \,j\neq k,\,\,j,k\in\N.
			\end{align}
			
			\emph{Proof of Claim 4:}
			If $\Phi$ is strictly increasing, when $I_0\prec\Phi(I_0)$\footnote{Let $I=(a, b)$ and $J=(c, d)$ be two intervals. Here and in what follows, we use the notation $I\prec J$ if $b<c$ and $I\succ J$ if $d<a$.}, it follows that
			\begin{align}\label{equ2.45.1}
				I_0\prec\Phi(I_0)\prec\Phi^2(I_0)\prec\cdots,
			\end{align}
			and when $I_0\succ\Phi(I_0)$, one has
			\begin{align*}
				I_0\succ\Phi(I_0)\succ\Phi^2(I_0)\succ\cdots.
			\end{align*}
			Thus \eqref{equ2.49} follows at once.
			On the other hand, if $\Phi$ is strictly decreasing, we first consider the case $I_0\prec\Phi^2(I_0)$. Note that $\Phi^2$ is strictly increasing,  then for any $m, n\in \mathbb{N}^+$, it follows that
			\begin{align}\label{equ2.50}
				I_0\prec\Phi^2(I_0)\prec\cdots\prec\Phi^{2n}(I_0) \quad\,\,\,\mbox{and}\quad\,\,\, \Phi^{2m+1}(I_0)\prec\cdots\prec\Phi^3(I_0)\prec\Phi(I_0).
			\end{align}
			If $\Phi(I_0)\prec I_0$, then the claim follows from \eqref{equ2.50} directly. If $I_0\prec\Phi(I_0)$, since $\Phi$ is strictly decreasing and $\Phi^2$ is strictly increasing, we have $I_0\prec\Phi^2(I_0)\prec\Phi^3(I_0)\prec\Phi(I_0)$. Repeating this procedure, we find that
			for any $m, n\in \mathbb{N}^+$,
			\begin{align*}
				I_0\prec\Phi^2(I_0)\prec\cdots\prec\Phi^{2n}(I_0)\prec\Phi^{2m+1}(I_0)\prec\cdots\prec\Phi^3(I_0)\prec\Phi(I_0),
			\end{align*}
			which implies \eqref{equ2.49}. The argument of the case $I_0\succ\Phi^2(I_0)$ is the same. Therefore the proof of the claim is complete.
			
			Since $T_1$ is a bijection  from $\R\to\R$,  it follows from  \textbf{Claim 4}  that
			\begin{align}\label{equ2.51}
				T_1\circ\Phi^j(I_0)\cap T_1\circ\Phi^k(I_0)=\emptyset,\,\quad\,\mbox{for all}\,\,\,  j\neq k,\, j,k\in\N.
			\end{align}
			Fix some $0\ne s_0\in C_0^\infty(\R)$ with $\mbox{supp}\,s_0\subset T_1(I_0)$, and  we define
			\begin{align*}
				s(\eta)&:=\left\{
				\begin{array}{ll}
					s_0((T_2T_1^{-1})^k(\eta)),\,\,\eta\in T_1\Phi^k(I_0),\,k=0,1,2,\dots ,m-1,\\[0.2cm]
					0,\qquad\quad\qquad\quad\eta\in\R\setminus\mathop{\cup}\limits_{k=0}^{m-1}T_1\Phi^k(I_0),
				\end{array}
				\right.\\[0.3cm]
				h(\xi)&:=\left\{
				\begin{array}{ll}
					-s_0(T_1\Phi^{-k}(\xi)),\,\,\,\,\xi\in\Phi^k(I_0),\,k=1,2,\dots ,m,\\[0.2cm]
					0,\qquad\qquad\quad\quad\xi\in\R\setminus\mathop{\cup}\limits_{k=1}^{m}\Phi^k(I_0),
				\end{array}
				\right.
			\end{align*}
			where $m$ is the same as in \emph{Case 1}  and we set
			\begin{align*}
				f:=\mathcal{F}^{-1}h\quad\,\,\mbox{and}\quad\,\,g:=\mathcal{F}^{-1}s.
			\end{align*}
			In view of  \eqref{equ2.49} and  \eqref{equ2.51},  both $h(\xi)$ and  $s(\eta)$ are well-defined and $h(\xi), s(\eta)\in C_0^\infty(\R)$. Further, a direct computation yields that
			\begin{align*}
				\hat g(T_1\xi)&=\left\{
				\begin{array}{ll}
					s_0(T_1\Phi^{-k}(\xi)),\quad\,\,\xi\in\Phi^k(I_0),\, \,\,k=0,1,2,\dots ,m-1,\\[0.2cm]
					0,\quad\qquad\quad\quad\quad\xi\in\R\setminus\mathop{\cup}\limits_{k=0}^{m-1}\Phi^k(I_0).
				\end{array}
				\right.\\[0.3cm]
				\hat g(T_2\xi)&=\left\{
				\begin{array}{ll}
					s_0(T_1\Phi^{-k}(\xi)),\quad\,\,\xi\in\Phi^k(I_0),\,k=1,2,\dots ,m,\\[0.2cm]
					0,\quad\qquad\quad\quad\quad\xi\in\R\setminus\mathop{\cup}\limits_{k=1}^{m}\Phi^k(I_0).
				\end{array}
				\right.
			\end{align*}
			This, together with the definition of $f$, implies that
			\begin{align}\label{equ2.52}
				\left\{
				\begin{array}{ll}
					\hat f(\xi)+\hat g(T_1\xi)=0,\,\,\forall\xi\in\R\setminus(I_0\cup\Phi^m(I_0)),\\[0.2cm]
					\hat f(\xi)+\hat g(T_2\xi)=0,\,\,\forall\xi\in\R.
				\end{array}
				\right.
			\end{align}
			Recall that $I_0$ satisfies the \emph{Case 1}, i.e.,  $I_0\subset\left(\overline{\pi_1(\Lambda_1)}\right)^c$ and $\Phi^m(I_0)\subset\left(\overline{\pi_1(\Lambda_1)}\right)^c$, then we have
			\begin{align}\label{equ2.53}
				\pi_1(\Lambda_1)\subset\R\setminus (I_0\cup\Phi^m(I_0)).
			\end{align}
			Combining \eqref{equ2.52}  and \eqref{equ2.53}, we obtain \eqref{equ2.12.0}.
			In conclusion, we have constructed some $\mu\neq0$ with $\hat\mu|_{\Lambda}=0$. Therefore ($\Gamma_0,\Lambda$) is not a Heisenberg uniqueness pair.
		\end{proof}
%
		
		\subsection{Examples and extensions of Theorem \ref{thm-2}}\label{extension1} We first give  some examples which are corollaries of theorem \ref{thm-2}. 
		
		\begin{example}\label{ex-1}
			If $\Lambda_1,\,\Lambda_2$ are two straight line that are symmetric about the axis, then $(\Gamma_0,\Lambda_1\cup \Lambda_2)$ is not a \emph{HUP}.
		\end{example}
		\begin{proof}
			Let $\Lambda_i=\{(\xi, k_i\xi),\,\,\xi\in\R\}$, $i=1,2$, where $k_1=-k_2\neq 0$, then
			\begin{align*}
				\Phi^2(\xi)=\xi,\,\,\,\,\mbox{for all}\,\,\xi\in\R.
			\end{align*}
			Thus $(\Gamma_0,\Lambda_1\cup \Lambda_2)$ is not a \emph{HUP} by Theorem \ref{thm-2} \textbf{Part (A)}.
		\end{proof}
%
In view of the invariance property \eqref{equ1.2.2}, we can reformulate the above result  as follows: consider
\begin{align*}
\Gamma=\{(\tau, \xi)\in \R^2,\, \tau^2- \xi^2=0\}\quad\text{and}\quad L_i=\{(x, k_ix),\,x\in\R\}.
\end{align*}
 Then $(\Gamma, L_1\cup L_2)$ is \textbf{not}  a \emph{HUP} when $k_1k_2=1$.
 This is consistent with the non-uniqueness results of 1-d wave equations. Indeed, arbitrarily fix a nonzero even  Schwartz function $h$, then $u(x,t)=h(\frac{1-k_1}{1+k_1}(x+t))-h(x-t)$ solves the wave equation
 $u_{tt}-u_{xx}=0$.
However, we deduce that  $u(x,k_1x)=u(x,k_2x)=0$ provided $k_1k_2=1$, i.e., $u(x, t)|_{L_1\cup L_2}=0$.

		\begin{example}\label{ex-2}
			Let $T_1\xi=k_1\xi,\,T_2\xi=k_2\xi$ where $0< k_1< k_2$ and $\lambda=\frac{k_2}{k_1}>1$. Let $\Lambda=\Lambda_1\cup\Lambda_2$ and arbitrarily fix  $a_0>0$, $b_0<0$.
			
			\noindent $(i)$  If $\Lambda_1=\{(\xi,T_1\xi):\xi\in\R\setminus(a_0,a)\}$ with some $a>\lambda a_0$, and $\Lambda_2=\{(\xi,T_2\xi):\xi\in\R\}$, then $(\Gamma_0,\Lambda)$ is not a \emph{HUP}.
			
			\noindent $(ii)$   If $\Lambda_1=\{(\xi,T_1\xi):\xi\in\R\setminus(E^+\cup E^-)\}$ and $\Lambda_2=\{(\xi,T_2\xi):\xi\in\R\}$, where
			\begin{align}\label{equ3.1}
				E^+=\mathop{\cup}\limits_{j=0}^{\infty}(\lambda^ja_j,\lambda^ja_{j+1}),\,\,\,E^-=\mathop{\cup}\limits_{j=0}^{\infty}(\lambda^jb_{j+1},\lambda^jb_j),
			\end{align}
			and $a_n=\frac{1}{2}(a_{n-1}+\lambda a_0)$, $b_n=\frac{1}{2}(b_{n-1}+\lambda b_0),\,n\in\N^\ast$.
			Then $(\Gamma_0,\Lambda)$ is a \emph{HUP}.

		\end{example}
		\begin{proof}
			Recall that $\Phi=T_2^{-1}\circ T_1$, thus $\Phi\xi=\frac{1}{\lambda}\cdot\xi$. Let $I=(\lambda a_0,a)\subset\R\setminus\overline{\pi_1(\Lambda_1)}$. Then
			\begin{align*}
				\Phi(I)=(a_0,\lambda^{-1}a)\subset(a_0,a)=\R\setminus\overline{\pi_1(\Lambda_1)},
			\end{align*}
			which implies that $\left(\overline{\pi_1(\Lambda_1)}\cap\overline{\pi_1(\Lambda_2)}\right)^c$ is a non-wandering set associate with $\Phi$. Then statement \noindent $(i)$  follows by Theorem \ref{thm-2},\textbf{ Part (B)}.
			
			Now we prove  \noindent $(ii)$. Since $\pi_1(\Lambda_2)=\R$, it follows that  $\pi_1(\Lambda)$ is dense in $\R$, moreover,   in order to prove that $\left(\overline{\pi_1(\Lambda_1)}\cap\overline{\pi_1(\Lambda_2)}\right)^c=E^+\cup E^-$ is a wandering set associate with $\Phi$, it suffices to prove that for any open interval $I\subset E^+\cup E^-$,
			\begin{align}\label{equ3.2}
			\Phi^k(I)\subset \overline{\pi_1(\Lambda_1)},\,\,\,\,\mbox{for all}\,\,\, k\in \mathbb{Z}\setminus\{0\}.
			\end{align}
			We first treat $E^+$. For any $I^{+}_j=(\lambda^ja_j,\lambda^ja_{j+1})$ and $k\in \mathbb{Z}\setminus\{0\}$, we have
			\begin{equation*}
				\Phi^k(I^{+}_j)=(\lambda^{j-k}a_j,\lambda^{j-k}a_{j+1}).
			\end{equation*}
			It follows from the definition of $a_n$ that
			\begin{equation}\label{equ3.3}
				a_n=\lambda a_0-(\frac{1}{2})^n(\lambda-1)a_0.
			\end{equation}
			From this we see that the sequence $\{a_n\}$ is strictly increasing and it's uniformly bounded since $a_n<\lambda\cdot a_0$. A direct computation shows that
			\begin{align*}
				\left\{
				\begin{array}{ll}
					\lambda^{j-k-1}a_{j-k}&< \lambda^{j-k}a_j<\lambda^{j-k}a_{j+1}\leq \lambda^{j-k}a_{j-k},\qquad\,\,\,\,\mbox{if}\,\,k<0,  \\
					\lambda^{j-k}a_{j-k+1}&\leq\lambda^{j-k}a_j<\lambda^{j-k}a_{j+1}<\lambda^{j-k+1}a_{j-k+1},\,\,\,\,\,\mbox{if}\,\,0<k\le j,\\
					&\lambda^{j-k}a_j<\lambda^{j-k}a_{j+1}<a_0,\qquad\qquad\,\quad\,\,\,\,\,\mbox{if}\,\,k>j.
				\end{array}
				\right.
			\end{align*}
			This shows that
			\begin{align*}
				\left\{
				\begin{array}{ll}
					I_{j-k-1}^+&\prec\Phi^k(I^{+}_j)\prec I_{j-k}^+,\quad\,\,\,\,\mbox{if}\,\,k<0,  \\
					I_{j-k}^+&\prec\Phi^k(I^{+}_j)\prec I_{j-k+1}^+,\,\,\,\,\,\mbox{if}\,\,0<k\le j,\\
					\Phi^k(I^{+}_j)&\subset(0, a_0)\qquad\qquad\,\,\,\,\,\,\mbox{if}\,\,k>j,
				\end{array}
				\right.
			\end{align*}
			where the notation $I\prec J$ is the same as in \eqref{equ2.45.1}. In particular, this shows that \eqref{equ3.2} holds with $I$ replaced by any $I^{+}_j$. Since the sequence $\{b_n\}$ has the same form as $\{a_n\}$ except that $b_0<0$, it follows that  \eqref{equ3.2} holds with $I$ replaced by any $I^{-}_j$ (defined by $I^{-}_j=(\lambda^jb_{j+1},\lambda^jb_j)$). Therefore we have shown that  $E^+\cup E^-$ is a wandering set associate with $\Phi$, and statement \noindent $(ii)$  follows by Theorem \ref{thm-2},\textbf{ Part (B)}.
		\end{proof}

		\begin{remark}\label{rmk-2}
			If we further decompose $ E^{\pm}$ into two components, $E^{\pm}=E^{\pm}_1\cup E^{\pm}_2$, where
			\begin{align*}
				E_1^+=\mathop{\cup}\limits_{j=0}^{\infty}(\lambda^ja_j,\lambda^j(a_j+a_{j+1})/2),\,\,\,\,\mbox{and}\,\,E_1^-=\mathop{\cup}\limits_{j=0}^{\infty}(\lambda^jb_{j+1},\lambda^j(b_{j+1}+b_j)/2).
			\end{align*}
			Let
			\begin{align}\label{equ3.4}
				\Lambda_j=\{(\xi,T_j\xi):\xi\in\R\setminus(E_j^+\cup E_j^-)\} ,\,\,\,\,\mbox{where}\,\,j=1, 2.
			\end{align}
			%
			Then a slight modification of the arguments above, we see that $(\Gamma_0, \Lambda_1\cup \Lambda_2)$ is a \emph{HUP}.
			
			Note that by \eqref{equ3.3}, one has
			\begin{equation*}
				|I_j^{+}|=\lambda^j\cdot(a_{j+1}-a_j)=\frac{\lambda-1}{2}\cdot (\frac{\lambda}{2})^ja_0,
			\end{equation*}
			which implies that $\displaystyle\lim_{j\to\infty}|I_j^{+}|=\infty$ if $\lambda>2$. This shows that $\pi_1(\Lambda_j)^c$ ($j=1, 2$) can contain increasing gaps of arbitrarily large size. As a comparison, we see from the  example \noindent $(i)$ that, even $\pi_1(\Lambda_1)^c$ is bounded, it may fail to be a \emph{HUP}, see Figure \ref{fig3} where  we suppose $T_1\xi=\frac{1}{2}\xi,\,T_2\xi=2\xi$.
			\begin{figure}
				\centering
				\subfigure[Example\ref{ex-2}- \noindent $(i)$]{
					\begin{tikzpicture}[scale=0.32]
						\draw[->](-1,0)--(16,0) node[right]{$\xi_1$};
						\draw[->](0,-1)--(0,8) node[above]{$\xi_2$};
						\draw[blue](-1/4,-1/8)--(1/4,1/8);
						\draw[blue](2,1)--(15,15/2) node[above]{$\Lambda_1$};
						\draw[red](-7/16,-7/8)--(4,8) node[above]{$\Lambda_2$};
						\draw[thin,dotted](1/4,0)--(1/4,1/2);
						\draw[thin,dotted](2,0)--(2,4);
						\draw[thin,dotted](1,0)--(1,2);
						\draw[thin,dotted](1/2,1)--(2,1);
						\draw[black](1/4,1/2)--(1/2,1);
						\draw[thick](1,0)--(1,0.05)node[below,outer sep=2pt,font=\tiny]at(1,0){$\lambda a_0$};
						\draw[thick](2,0)--(2,0.05)node[below,outer sep=2pt,font=\tiny]at(2,0){$a$};
					\end{tikzpicture}
				}
				\hspace{0.2in}
				\subfigure[Example\ref{ex-2}- $(ii)$]{
					\begin{tikzpicture}[scale=0.27]
						\draw[->](-1,0)--(16,0) node[right]{$\xi_1$};
						\draw[->](0,-1)--(0,8) node[above]{$\xi_2$};
						\draw[blue](-1/4,-1/8)--(1/4,1/8);
						\draw[blue](5/8,5/16)--(5/2,5/4);
						\draw[blue](13/4,13/8)--(13,13/2);
						\draw[blue](29/2,29/4)--(17,17/2) node[above]{$\Lambda_1$};
						\draw[red](-7/16,-7/8)--(4,8) node[above]{$\Lambda_2$};
						\draw[thin,dotted](1/4,0)--(1/4,1/2);
						\draw[thin,dotted](1/4,1/2)--(1,1/2);
						\draw[thin,dotted](1,2)--(1,1/2);
						\draw[thin,dotted](5/8,5/4)--(5/2,5/4);
						\draw[thin,dotted](5/8,0)--(5/8,5/4);
						\draw[thin,dotted](13/4,0)--(13/4,13/2);
						\draw[thin,dotted](5/2,0)--(5/2,5);
						\draw[thin,dotted](13,13/2)--(13/4,13/2);
						\draw[thin,dotted](10,5)--(5/2,5);
						\draw[thick](1/4,0)--(1/4,0.05)node[below,outer sep=2pt,font=\tiny]at(1/4,0){$a_0$};
						\draw[thick](5/8,0)--(5/8,0.05)node[below,outer sep=2pt,font=\tiny]at(9/8,0){$a_1$};
						\draw[thick](5/2,0)--(5/2,0.05)node[below,outer sep=2pt,font=\tiny]at(5/2,0){$\lambda a_1$};
						\draw[thick](13/4,0)--(13/4,0.05)node[below,outer sep=2pt,font=\tiny]at(15/4,0){$\lambda a_2$};
					\end{tikzpicture}
				}
				\caption{Example \ref{ex-2}}\label{fig3}
			\end{figure}
		\end{remark}

Next, we remark that the assumption of 	$\Phi$ (a continuous bijection from $\R\to\R$) can be relaxed a bit, Indeed, we can combine Theorem \ref{thm-1} with  Theorem \ref{thm-2} to prove the following
		\begin{corollary}\label{pro3}
			Let $a, b>0$ and $T_1,\,T_2:\,(-a,a)\to(-b,b)$ be continuous bijections. Let $\Lambda=\Lambda_1\cup \Lambda_2$, where
			\begin{equation*}
				\Lambda_1=\{(\xi,T_1\xi):\xi\in(-a,a)\}\cup\{(\xi,-b):\xi\in(-\infty,-a]\}\cup\{(\xi,b):\xi\in[a,\infty)\}
			\end{equation*}
			and
			\begin{equation*}
				\Lambda_2=\{(\xi,T_2\xi):\xi\in(-a,a)\}\cup\{(-a,\eta):\eta\in[b,\infty)\}\cup\{(a,\eta):\eta\in(-\infty,b]\}
			\end{equation*}
			\noindent $(\mathbf{A})$ If there exists an open interval $I\subset(-a,a)$, such that \eqref{equ1.10} holds,  then $(\Gamma_0,\Lambda)$ is not a Heisenberg uniqueness pair.
			
			\noindent $(\mathbf{B})$ Assume that  \eqref{equ1.10} doesn't hold for any open interval $I\subset (-a, a)$. Then $(\Gamma_0,\Lambda)$ is a Heisenberg uniqueness pair if and only if \noindent $(i)$: $\pi_1(\Lambda),\,\pi_2(\Lambda)$ are dense in $\R$; \noindent $(ii)$: $(-a,a)\setminus\left(\overline{\pi_1(\Lambda_1)}\cap\overline{\pi_1(\Lambda_2)}\right)$ is wandering set associate with $\Phi=T_2^{-1}\circ T_1$.
		\end{corollary}
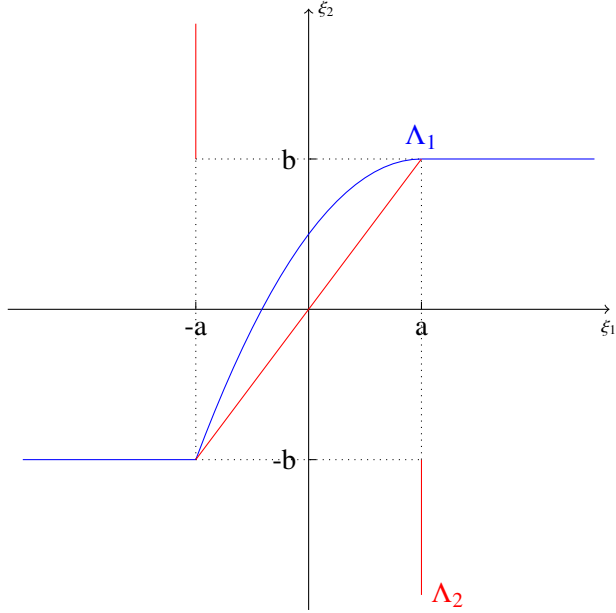
\begin{figure}
			\centering
			\begin{tikzpicture}
				\draw[->](-4,0)--(4,0)node[left,below,font=\tiny]{$\xi_1$};
				\draw[->](0,-4)--(0,4)node[right,font=\tiny]{$\xi_2$};
				\draw[blue](-3.8,-2)--(-1.5,-2);
				\draw[blue](3.8,2)--(1.5,2);
				\draw[red](-1.5,2)--(-1.5,3.8);
				\draw[red](1.5,-2)--(1.5,-3.8)node[right]{$\Lambda_2$};
				\draw[thin,dotted](-1.5,-2)--(-1.5,2);
				\draw[thin,dotted](1.5,-2)--(1.5,2);
				\draw[thin,dotted](-1.5,2)--(1.5,2);
				\draw[thin,dotted](-1.5,-2)--(1.5,-2);
				\draw(-1.5,0)--(-1.5,0.1)node[below=3.6pt]{-a};
				\draw(1.5,0)--(1.5,0.1)node[below=3.6pt]{a};
				\draw(0,-2)--(0.1,-2)node[left=3.6pt]{-b};
				\draw(0,2)--(0.1,2)node[left=3.6pt]{b};
				\draw[color=blue,domain=-1.5:1.5]plot(\x,{-4*(\x-1.5)^2/9+2})node[above]{$\Lambda_1$};
				\draw[color=red,domain=-1.5:1.5]plot(\x,{4*(\x)/3});
			\end{tikzpicture}
			\caption{Corollary \ref{pro3} }\label{fig4}
		\end{figure}
		\begin{proof}
			The proof of part \noindent $(\mathbf{A})$ follows immediately from Theorem \ref{thm-2}-part \noindent $(\mathbf{A})$.
			
			The proof of part \noindent $(\mathbf{B})$ follows from a minor modification of Theorem \ref{thm-1} and  Theorem \ref{thm-2}-part \noindent $(\mathbf{B})$.
			Indeed, we first consider the region $\xi_1\ge a$, observing that $\hat{\mu}|_{\Lambda}=0$ implies
			\begin{align}\label{equ3.5}
				\left\{
				\begin{array}{ll}
					\hat{f}(\xi_1)+\hat{g}(b)=0,\quad\,\,\, \,\,\mbox{for all}\,\,\,a\leq \xi_1\in\pi_1(\Lambda_1),\\[0.3cm]
					\hat{f}(a)+\hat{g}(\xi_2)=0,\quad \,\,\,\,\,\mbox{for all}\,-b\ge \xi_2\in\pi_2(\Lambda_2).
				\end{array}
				\right.
			\end{align}
			Parallel to \eqref{equ2.3}-\eqref{equ2.5} in Theorem  \ref{thm-1}, we deduce that
			\begin{equation}\label{equ3.6}
				\hat{f}(\xi)=-\hat g(b)=0, \,\,\,\,\xi\ge a \quad\mbox{and}\quad  \hat{g}(\xi)=-\hat f(a)=0, \,\,\,\,\xi\leq -b.
			\end{equation}
			Similarly, in the region $\xi_1\le -a$, we have
			\begin{equation}\label{equ3.7}
				\hat{f}(\xi)=-\hat g(-b)=0, \,\,\,\,\xi\leq a \quad\mbox{and}\quad  \hat{g}(\xi)=-\hat f(-a)=0, \,\,\,\,\xi\ge b.
			\end{equation}
			Finally, when $-a<\xi_1<a$, we can repeat the proof of  Theorem \ref{thm-2}-part \noindent $(\mathbf{B})$ except for the following differences:\\
			\noindent $(i)$ $J=\{\xi\in(-a,a):\Phi^2(\xi)=\xi\}$ is a nowhere dense set and
			\begin{align*}
				\displaystyle\lim_{n\to\pm\infty}\Phi^{2n}(\xi)\in J\cup\{a,-a\},\,\,\,\,\,\mbox{for all}\,\,\,\xi\in(-a,a).
			\end{align*}
			\noindent $(ii)$  By \eqref{equ3.6} and \eqref{equ3.7}, $\hat{f}$ and $\hat{g}$ vanish at end points $\pm a$ and $\pm b$ respectively.
		\end{proof}

		\section{Higher dimensional Heisenberg Uniqueness Pairs--Proof of Theorem \ref{thm-3}}\label{section3}
	
		We follow the approach in \cite{GJ}. Let $\pi^\perp$ be the orthogonal projection on $\text{span}(u_1, u_2)^\perp$. There exists a measurable subset $S_0\subset S$ such that the decomposition $y=x+u,\, x\in S_0,\, u\in \text{span}(u_1, u_2)$ is unique for every $y\in S$, and we write $x=\tilde{\pi}^\perp(y)$. Using disintegration theorem (see \cite[p. 10]{GJ}), one has
		\begin{align}\label{equ3.13}
			\int_{S}\phi(x)d\mu(x)=\int_{S_0}\int_{(\widetilde{\pi}^\perp)^{-1}(x)}\phi(y)d\mu_x(y)d\nu(x),\,\quad\,\phi\in C_0(S).
		\end{align}
		Moreover $\mu_x$ is supported on
		\begin{align}\label{equ3.11}
			S_x:=S\cap(x+span(u_1,v_2))=\{x+su_1+tv_2:(s,t)\in\Sigma\},
		\end{align}
		where $\Sigma$ be the set of points $(s,t)\in\R^2$ satisfying
		\begin{align}\label{equ3.12}
			As^2+2Bst+Ct^2+2Ds+2Et=0.
		\end{align}
		According to \cite[Lemma 2.4]{GJ}, one has
		\begin{align}\label{equ3.10}
			\left.\widehat{\mu}\right|_{H_{u_1}\cup H_{u_2}}=0 \Longleftrightarrow \left.\widehat{\mu}_x\right|_{\R u_1\cup \R u_2}=0
		\end{align}
		for $\nu$-almost all $x\in S_0$.
		Combining \eqref{equ3.10}, \eqref{equ3.13} and the fact $supp\,\mu_x\subset S_x$, it follows that
		\begin{align}\label{equ3.14}
			(S,H_{u_1}\cup H_{u_2})\mbox{ is a \emph{HUP}}\Leftrightarrow (S_x,\R u_1\cup\R u_2) \mbox{ is a \emph{HUP} for }\nu\mbox{-almost all }x\in S_0.
		\end{align}
		
		As is observed in \cite{GJ}, $S_x$ is a conic section and its type is determined by the discriminant of the quadratic form in \eqref{equ3.12}, i.e., an ellipse (when $AC-B^2>0$), or parabola (when $AC-B^2=0$), hyperbola (when $AC-B^2<0$).  \eqref{equ3.14} reduces the study in two dimensional case, and we divide the analysis into three cases.
		
		\emph{Case 1}:
		$AC-B^2=0$. It follows from \eqref{equ3.12} that $S\cap(x+span(u_1,v_2))$ is a parabola for $\nu-$ almost all $x\in S_0$.
		It follows from \cite[Theorem 1]{Sj2} that $(\Gamma,\Lambda)$ is a \emph{HUP} when $\Gamma$ is a parabola and $\Lambda$ is the union of two different straight lines.
		Therefore, $(S\,,\,H_{u_1}\cup H_{u_2})$ is a \emph{HUP}.
		
		\emph{Case 2}:
		$AC-B^2<0$. In this case,  $S\cap(x+span(u_1,v_2))$ is a hyperbola for $\nu-$ almost all $x\in S_0$. Using the invariance property \eqref{equ1.2.1}, it suffices to consider
		\begin{align}\label{equ3.12'}
			(As+Bt)^2-(B^2-AC)t^2=D^2-\frac{(AE-BD)^2}{B^2-AC}, \tag{\ref{equ3.12}'}
		\end{align}
		where $D^2(B^2-AC)-(AE-BD)^2\ne 0$ (the non-degenerate case). Under the rotation transformation
		\begin{equation}\label{equ3.17}
			\left\{
			\begin{array}{ll}
				s'=\cos\varphi\cdot s+\sin\varphi\cdot t,\\[0.15cm]
				t'=-\sin\varphi\cdot s+\cos\varphi\cdot t,
			\end{array}
			\right.
		\end{equation}
		with $\varphi$ satisfying
		\begin{align}\label{equ3.18}
			\tan2\varphi=\frac{2B}{A-C}.
		\end{align}	
		We rewrite \eqref{equ3.12'} in the following standard  form
		\begin{align}\label{equ3.12''}	
			g(\varphi)\cdot s'^2 + g(\varphi+\frac{\pi}{2})\cdot t'^2=1
		\end{align}
		with
		\begin{equation}\label{equ3.24}
			g(\varphi)=\frac{(B^2-AC)[(A\cos\varphi+B\sin\varphi)^2-(B^2-AC)\sin^2\varphi]}{D^2(B^2-AC)-(AE-BD)^2}
		\end{equation}
		and $g(\varphi)g(\varphi+\frac{\pi}{2})<0$,
		Notice that \eqref{equ3.18} has two solutions $\varphi_1\in [0,\pi/2)$ or $\varphi_2=\varphi_1+\frac{\pi}{2}$. Thus we can choose one (still denoted by $\varphi$ for convenience) such that
		\begin{align}\label{equ3.25}
			a_0^{-2}:=g(\varphi)>0,\qquad  b_0^{-2}:=-g(\varphi+\frac{\pi}{2})>0.
		\end{align}
		Inserting \eqref{equ3.25} into \eqref{equ3.12''}	yields
		\begin{align}\label{equ3.12'''}
			S_x=\left\{(s',t')\in\R^2:\,\frac{s'^2}{a_0^2}-\frac{t'^2}{b_0^2}=1\right\}. \tag{\ref{equ3.12}\mbox{''}}
		\end{align}
		By \eqref{equ3.17} and the fact \eqref{equ1.20}, we notice that
		\begin{equation}\label{equ3.26}
			\R u_1=l_{-\varphi}\quad\,\mbox{and}\quad\,  \R u_2=l_{-(\varphi-\theta_0)}
		\end{equation}
		in the coordinate system generated by $s',\,t'$.
		
		Consider $\Lambda=l_{\theta_1}\cup l_{\theta_2}$ where $l_\theta=\{(t\cos\theta,t\sin\theta):t\in\R\},\,0\le\theta<\pi$. We recall the following result due to Jaming and Kellay.
		\begin{theorem}\cite[Theorem 3.5]{JK} \label{thm3.5}
			Let $S_x$ be given by \eqref{equ3.12'''}.\footnote{We remark that Theorem 3.5 in \cite{JK} is proved  for $a_0=b_0=1$,  the general case follows from the invariance property \eqref{equ1.2.2}. \noindent(iii) is not stated in \cite{JK}, but the technique proved for  \noindent(ii) also applies with a minor modification. }
			
			\noindent(i) If $\tan\theta=\pm a_0/b_0$, then $(S_x, l_\theta)$ is a \emph{HUP}.
			
			\noindent(ii) If $\theta_1\neq\theta_2\in(0,\pi)$, then $(S_x, \Lambda)$ is a \emph{HUP} if and only if $\tan\theta_1\cdot\tan\theta_2\neq a_0^2/b_0^2$.
			
			\noindent(iii) If $\theta_1=0$, then $(S_x, \Lambda)$ is a \emph{HUP} if and only if $\theta_2\neq\frac{\pi}{2}$
		\end{theorem}
		We shall apply Theorem \ref{thm3.5} to show that
		\begin{equation}\label{equ3.18.1}
			(S_x, \R u_1\cup\R u_2))\mbox{ is a \emph{HUP}}\Longleftrightarrow B(v_1, v_2)\ne 0.
		\end{equation}
		
		In view of \noindent(i)--\noindent(iii) above, the proof is divided into tree subcases.
		
		
		\emph{Subcase 1: $\tan\varphi=\pm a_0/b_0\,\text{or}\,\tan(\varphi-\theta_0)=\pm a_0/b_0$.} In view of \noindent(i) in Theorem \ref{thm3.5}, it suffices to prove that $B(v_1, v_2)\ne 0$.
		
		When $\tan\varphi=\pm a_0/b_0$ (which implies that $\varphi\ne 0,\frac{\pi}{2}$). By \eqref{equ3.25}, it follows that
		\begin{equation*}
			\tan^2\varphi=-\frac{C-2B\tan\varphi+A\tan^2\varphi}{A+2B\tan\varphi+C\tan^2\varphi}.
		\end{equation*}
		Note that $\tan2\varphi$ satisfies \eqref{equ3.18} and use the fact $\varphi\ne 0,\frac{\pi}{2}$, we obtain that  \noindent(i)  $B\ne 0$; \noindent(ii)  $\tan^2\varphi= a^2_0/b^2_0$ if and only if $C=0$. These, together with the assumption  $\theta_0\in (0, \pi)$,  imply that $B(v_1, v_2):=C\cos\theta_0-B\sin\theta_0=-B\sin\theta_0\ne 0$.
		
		When $\tan(\varphi-\theta_0)=\pm a_0/b_0$. By \eqref{equ3.25} we have
		\begin{equation*}
			\left(\frac{\sin\varphi\cos\theta_0-\cos\varphi\sin\theta_0}{\cos\varphi\cos\theta_0+\sin\varphi\sin\theta_0}\right)^2
			=-\frac{A\sin^2\varphi-2B\sin\varphi\cos\varphi+C\cos^2\varphi}{A\cos^2\varphi+2B\sin\varphi\cos\varphi+C\sin^2\varphi}.
		\end{equation*}
		This implies that $\cos\theta_0(C\cos\theta_0-B\sin\theta_0)+\sin\theta_0(A\sin\theta_0-B\cos\theta_0)=0$. We prove by contradiction that $B(v_1, v_2)\neq0$. In fact, if $B(v_1, v_2)=C\cos\theta_0-B\sin\theta_0=0$, thus $A\sin\theta_0-B\cos\theta_0=0$, combining the two equation yields that $B^2=AC$, which is a contradiction.
		

		\emph{Subcase 2: $B\neq0$ and $\theta_0\neq\varphi$}.  By \eqref{equ3.18}, we have $\varphi\neq0$ and $\frac{\pi}{2}$.  Apply Theorem \ref{thm3.5} \noindent(ii) with $\theta_1=\pi-\varphi$, $\theta_2=\theta_0-\varphi$, we deduce that  $(S_x, \R u_1\cup\R u_2)$ is a \emph{HUP} if and only if $\tan\varphi\cdot\tan(\varphi-\theta_0)\neq a_0^2/b_0^2$.
		Therefore, it suffices to  prove that
		\begin{equation}\label{equ3.20}
			\tan\varphi\cdot\tan(\varphi-\theta_0)= a_0^2/b_0^2 \Longleftrightarrow B(v_1, v_2)=0.
		\end{equation}
		
		If $\cos\theta_0=0$, i.e., $\theta_0=\frac{\pi}{2}$, we have $v_1=-u_1$ and $v_2=u_2$ by \eqref{equ1.20}. Then $B(\nu_1,\nu_2)=-B\neq0$ by assumption. Meanwhile, $\tan\varphi\cdot\tan(\varphi-\theta_0)=-1\neq a_0^2/b_0^2$.  Thus \eqref{equ3.20} holds.
		
		If $\cos\theta_0\neq0$, using the expression \eqref{equ3.25}, we see that $\tan\varphi\cdot\tan(\varphi-\theta_0)= a_0^2/b_0^2$ if and only if
		\begin{align}\label{equ3.28}
			\frac{\tan^2\varphi-\tan\theta_0\tan\varphi}{1+\tan\theta_0\tan\varphi}=-\frac
			{(B\cos\varphi-A\sin\varphi)^2-(B^2-AC)\cos^2\varphi}{(A\cos\varphi+B\sin\varphi)^2-(B^2-AC)\sin^2\varphi}.
		\end{align}
		By \eqref{equ3.18} and the assumption $B\neq0$, we have
		\begin{equation}\label{equ3.29}
			1-\tan^2\varphi=\frac{A-C}{B}\tan\varphi.
		\end{equation}
		Inserting \eqref{equ3.29} into \eqref{equ3.28}, we find that  \eqref{equ3.28} holds if and only if
		\begin{equation}\label{equ3.30}
			C-B\cdot\tan\theta_0=0.
		\end{equation}
		On the other hand, by \eqref{equ1.20} and \eqref{equ3.8.1}, we have
		\begin{equation}\label{equ3.31.1}
			B(v_1,v_2)=B(\cos\theta_0v_2-\sin\theta_0u_1,  v_2)=\cos\theta_0(C-B\cdot\tan\theta_0).
		\end{equation}
		This, together with the assumption that $\cos\theta_0\neq0$, yields that \eqref{equ3.30} holds if and only if $B(v_1, v_2)=0$.  Therefore the proof of \eqref{equ3.20} is complete.
		
		\emph{Subcase 3: $B=0$ or $\theta_0=\varphi$}.  In this case, we have $\varphi=0,\,\frac{\pi}{2}$ or $\theta_0$. It follows from   Theorem \ref{thm3.5} \noindent(iii) above that $(S_x,\R u_1\cup\R u_2)$ is a \emph{HUP} if and only if $\theta_0\neq\frac{\pi}{2}$.
		Thus, it suffices to prove
		\begin{equation}\label{equ3.27}
			\theta_0=\frac{\pi}{2} \Longleftrightarrow B(v_1, v_2)=C\cos\theta_0-B\sin\theta_0=0.
		\end{equation}
		
		When $B=0$, recall that $AC-B^2<0$, hence $C\ne 0$ and  $B(v_1, v_2)=0$ if and only if $\cos\theta_0=0$, i.e., $\theta_0=\frac{\pi}{2}$.
		Therefore \eqref{equ3.27} holds.
		
		When $\theta_0=\varphi$, if $\theta_0=\frac{\pi}{2}$, then $\varphi=\frac{\pi}{2}$  and by \eqref{equ3.18} we have $B=0$, which implies that $B(v_1, v_2)=0$; Conversely, if $B(v_1, v_2)=0$, we prove $\theta_0=\frac{\pi}{2}$ by contradiction. Indeed, if $\theta_0\ne\frac{\pi}{2}$, then $\cos\varphi\neq0$ and $B\neq0$. Observe that in this case $B(v_1, v_2)=0$ if and only if $C\cos\varphi -B\sin\varphi =0$ i.e., $\tan\varphi=\frac{C}{B}$. However, this together with \eqref{equ3.18} yields that $B^2=AC$, which contradicts with the fact $AC-B^2<0$. Therefore \eqref{equ3.27} holds.
		
		Combining subcase 1, 2 and 3,  we complete the proof when $AC-B^2<0$.
		
		\emph{Case 3: $AC-B^2>0$.} $S\cap(x+span(u_1,u_2))$ is an ellipse for $\nu-$ almost all $x\in S_0$. In this case, we need the following result:
		
		\begin{theorem}\cite{JK,Le,Sj} \label{thm-3.1}
			Let
			$$\Gamma=\{(x_1,x_2)\in\R^2:\,\frac{x_1^2}{a^2}+\frac{x_2^2}{b^2}=1,\,a\ge b>0\}$$
			and $\Lambda=l_{\theta_1}\cup l_{\theta_2}$ where $l_\theta=\{(t\cos\theta,t\sin\theta):t\in\R\},\,0\le\theta<\pi$. Then $(\Gamma,\Lambda)$ is a \emph{HUP} if and only if
			\begin{align}\label{equ3.15}
				\arctan\left(\frac{b}{a}\tan\theta_2\right)-\arctan\left(\frac{b}{a}\tan\theta_1\right)\notin\pi\mathbb{Q}.
			\end{align}
		\end{theorem}
	
	As is shown in \emph{Case 2}, \eqref{equ3.12''} and \eqref{equ3.24} hold by using the rotation transform \eqref{equ3.17} with $\varphi$ satisfying  \eqref{equ3.18}, while in this case, there exists some $\varphi\in (0, \pi)\setminus \{\frac{\pi}{2}\}$  such that  $g(\varphi+\frac{\pi}{2})\ge g(\varphi)>0$, still denoted by $\varphi$ for convenience. Similarly, if we set $a^{-2}=g(\varphi)$, $b^{-2}=g(\varphi+\frac{\pi}{2})$, then we deduce that
	\begin{align*}
		S_x=\{(s',t')\in\R^2:\,\frac{s'^2}{a^2}+\frac{t'^2}{b^2}=1,\,a\ge b>0\},
	\end{align*}
moreover, by \eqref{equ3.24}, $a$ and $b$ satisfy
	\begin{equation}\label{equ3.19.1}	\frac{b^2}{a^2}=\frac{(A\cos\varphi+B\sin\varphi)^2-(B^2-AC)\sin^2\varphi}{(B\cos\varphi-A\sin\varphi)^2-(B^2-AC)\cos^2\varphi}=\frac{A+2B\tan\varphi+C\tan^2\varphi}{C-2B\tan\varphi+A\tan^2\varphi}.
	\end{equation}

	Finally, By  Theorem \ref{thm-3.1} and the fact \eqref{equ3.26}, we conclude that $(S_x,\R u_1\cup\R u_2)$ is a \emph{HUP} if and only if
	\begin{equation}\label{equ3.32}
		\arctan\left(\frac{b}{a}\tan(\theta_0-\varphi)\right)+\arctan\left(\frac{b}{a}\tan\varphi\right)\notin\pi\mathbb{Q},
	\end{equation}
where $\frac{b}{a}$ is given by \eqref{equ3.19.1}.
The proof of Theorem \ref{thm-3} is complete. \qed
	
	\begin{remark}\label{rmk3.2}
		Compared with  Theorem \ref{thm-2}, it seems natural to consider the following problem in higher dimensions: Let $\mu$ be a finite complex-valued measure supported in the cone $S$ (given by \eqref{equ3.8})  and assume that $\widehat{\mu}|_{\Lambda}=0$, where  $\Lambda\subset H_{u_1}\cup H_{u_2}$ is a subset, then  what's the minimum required information on the zero set $\Lambda$ such that $(S, \Lambda)$ forms a Heisenberg uniqueness pair?
		We mention that  the proof of Theorem \ref{thm-3} is based on the observation \eqref{equ3.10}, which may not hold if $\Lambda$ is only a subset of $H_{u_1}\cup H_{u_2}$. On the other hand, it's quite different from the two dimensional case since we don't have expressions like \eqref{equ2.1}-\eqref{equ2.2} and the tools there can't be applied directly.
	\end{remark}

	\noindent
	\section*{Acknowledgments}
	S. Huang was supported by the National Natural Science Foundation of China under grants 12171178 and 12171442.

	
\end{document}